\documentclass{article}

\RequirePackage{amsthm,amsmath,amssymb,amsfonts}
\RequirePackage[colorlinks,citecolor=blue,urlcolor=blue]{hyperref}

\usepackage[margin = 1.5in]{geometry}

\usepackage{setspace}
\theoremstyle{plain}
\newtheorem{theorem}{Theorem}
\newtheorem{lemma}{Lemma}
\theoremstyle{remark}
\newtheorem{rem}{Remark}
\newtheorem{definition}{Definition}
\newtheorem{assumption}{Assumption}
\newtheorem{example}{Example}

\newcommand{\lfrac}[2]{\genfrac{}{}{0pt}{}{#1}{#2}}

\begin{document}
\title{Non-central limit theorems for functionals of random fields on hypersurfaces.\footnote{Short title: Non-central asymptotics for random fields on hypersurfaces.} }
\date{}
\author{Andriy Olenko and Volodymyr Vaskovych}
\maketitle

\begin{center}{Department of Mathematics and Statistics,
		La Trobe University,\\ Melbourne, 3086, Australia \\
		a.olenko@latrobe.edu.au}
\end{center}

\vspace{6cm}
\section*{Acknowledgements}

Andriy Olenko was partially supported under the Australian Research Council's Discovery Projects funding scheme (project number  DP160101366) and the La Trobe University DRP Grant in Mathematical and Computing Sciences.
\newpage
\begin{abstract}
		This paper derives non-central asymptotic results for non-linear integral functionals of homogeneous isotropic Gaussian random fields defined on hypersurfaces in $\mathbb{R}^d$. We obtain the rate of convergence for these functionals. The results extend recent findings for solid figures. We apply the obtained results to the case of sojourn measures and demonstrate different limit situations. 
	\end{abstract}
{\bf Keywords:} Non-central limit theorems, Random field, Long-range dependence, Hermite-type distribution, Sojourn measures.\\
{\bf MSC:}  Primary 60G60; Secondary 60F05, 60G12
\section{Introduction} In this article we study real-valued homogeneous isotropic Gaussian random fields with long-range dependence. Long-range dependence is a well-established empirical phenomenon which appears in various fields, such as physics,  hydrology, signal processing, network traffic analysis, telecommunications, finance, econometrics, just to name a few.  See \cite{Douk}, \cite{Iv}, \cite{Wack} for more details. 

Various functionals of random fields have been a topic of interest in recent years, see, for example, \cite{Anh}, \cite{Ole}, \cite{LeoRMT}. In this research, we focus on non-linear integral functionals of Gaussian random fields defined on hypersurface sets. These functionals play an important role in various fields, for example, in cosmology, meteorology and  image analysis. It was shown, see \cite{Dob}, \cite{Taq1}, and \cite{Taq2}, that these functionals can produce non-Gaussian limits and require normalizing coefficients different from those in central limit theorems. For the more detailed overview of the problem, history of development, various approaches and existing results one can refer to \cite{New} and references therein.

In this research we use results from \cite{New}, \cite{Souj}, \cite{Bul} and obtain analogous asymptotics for the case of hypersurfaces. Most of the research conducted in this area considered only random fields defined on solid figures. Limit distributions for the functionals on spheres, which are a particular case of hypersurfaces, were studied in  \cite{Iv}. However, there were no results about the rate of convergence for the case of hypersurfaces. In this article we consider a general case of hypersurface sets. We are interested in both limit distributions, and rates of convergence to these limits. We prove that, analogously to the solid figure situation, the limit distribution is a Hermite-type distribution and it depends only on the Hermite rank of the integrands. However, while for all integrands with the same Hermite rank the limit distribution remains the same, we demonstrate that the rate of convergence can be different. To prove the results we need some fine geometric properties of hypersurfaces. Specifically, we use the rates of the average decay of the Fourier transform of surface measures, see \cite{Ios} and \cite{IosRud}.

Geometric properties of random fields on hypersurfaces are of interest in many applied areas, such as medical imaging, meteorology, and astrophysics. Many of these properties can be studied by the use of sojourn measures. Extensive literature is available concerning this topic, for some examples see \cite{Bul}, \cite{Adl}, \cite{Nov}, \cite{Mar}. Recently, non-Gaussian limits for the first Minkowski functional of random fields defined on 3-dimensional spheres were discussed in \cite{LeoNew}. In this article we obtain limits for sojourn measures of random fields defined on arbitrary hypersurfaces. We provide examples when these limits are Gaussian and Hermite-type of the rank 2, 3, and 4.

Various authors, see \cite{DavMart}, \cite{NourPol}, \cite{Zint} and the references therein, studied a distance between two Wiener-Ito integrals of the same rank. These result can be used to estimate the rate of convergence when the integrands are Hermite polynomials of Gaussian random fields. We estimate the Kolmogorov's distance between two Wiener-Ito integrals of the same rank and provide a small comparison of the existing results. 

The article is organized as follows. In Section~\ref{defs} we recall some basic definitions and assumptions that are required to present our main results. Section~\ref{abs} studies the asymptotic behavior of the considered functionals. Section~\ref{sojs} demonstrates how results from Section~\ref{abs} can be applied in the case of sojourn measures. Section~\ref{convs} provides the results on the rate of convergence.

\section{Definitions and assumptions}\label{defs}

In this section we provide main definitions and assumptions that are used in this work.

In what follows $|\cdot|$ and $\|\cdot\|$ denote the Lebesgue measure and the Euclidean distance in $\mathbb{R}^d$, $d\geq 2$, respectively. Let $B(y, s)$ be a $d$-dimensional ball with centre $y$ and radius $s$, and let $\rm{S}_{d-1}(r)$ be a sphere in $\mathbb{R}^{d}$ with the radius $r$.  We use the symbols $C$ and $\delta$ to denote constants which are not important for our exposition. Moreover, the same symbol may be used for different constants appearing in the same proof.

Let $\Delta$ be a bounded set in $\mathbb{R}^d,\, d\geq 2,$ with a boundary $\partial\Delta$. Let $\Delta(r)$, $r > 0,$ be the homothetic image of the set $\Delta$ with the centre of homothety at the origin and the
coefficient $r > 0$, that is $|\Delta(r)| = r^d|\Delta|$. Let $\partial\Delta$ be an Ahlfors-David regular hypersurface in~$\mathbb{R}^d$. One can find more information about Ahlfors-David regular sets in \cite{ReyBla} and references therein.

\begin{definition} {\rm\cite{ReyBla}}
	A closed hypersurface $\partial\Delta$ is called Ahlfors-David regular if there exists a constant $C$ such that for any $y\in\partial\Delta$ and $s >0$
	\begin{equation}\label{adr}
	C^{-1}s^{d-1} < \int\limits_{\partial\Delta\cap B(y, s)}\mathrm{d}\sigma(x) < Cs^{d-1},
	\end{equation}
	where $\mathrm{d}\sigma(\cdot)$ is the $d-1$-dimensional Lebesgue measure on the hypersurface set.
\end{definition}

Let $\Delta$ be a convex set, a polyhedron, or have a sufficiently smooth boundry, for example, from $C^{3/2}$ class. Let
\begin{equation*}
\mathcal{K}(x):=\int\limits_{\partial\Delta }e^{i<x,u>} \mathrm{d}\sigma(u),\quad
x\in\mathbb{R}^{d}.
\end{equation*}
In \cite{Ios} and \cite{IosRud} the rate of convergence was given for the average decay of the Fourier transforms $\mathcal{K}(\cdot)$
\begin{equation}\label{l2a}
\int\limits_{S_{d-1}(1)}|\mathcal{K}(\omega r)|^2\,\mathrm{d}\omega\leq Cr^{-d+1}.
\end{equation}  
In the discussion authors even hypothesised that this result should also hold for Lipschitz boundaries of compact sets, which is a much weaker condition.

The proof of the main results of our paper also remains valid for other hypersurfaces satisfying conditions (\ref{adr}) and (\ref{l2a}).

We consider a measurable mean-square continuous zero-mean homogeneous
isotropic real-valued random field $\eta (x),\ x\in \mathbb{R}^{d},$ defined
on a probability space $(\Omega ,\mathcal{F},P),$ with the covariance function 
\begin{equation*}
\text{\rm{B}}(r):=\mathrm{Cov}\left( \eta (x),\eta (y)\right)
=\int_{0}^{\infty }Y_{d}(rz)\,\mathrm{d}\Phi (z),\ x,y\in \mathbb{R}^{d},
\end{equation*}%
where $r:=\left\Vert x-y\right\Vert ,$ $\Phi (\cdot )$ is the isotropic
spectral measure, the function $Y_{d}(\cdot )$ is defined by 
\begin{equation*}
Y_{d}(z):=2^{(d-2)/2}\Gamma \left( \frac{d}{2}\right) \ J_{(d-2)/2}(z)\
z^{(2-d)/2},\quad z\geq 0,
\end{equation*}%
and $J_{(d-2)/2}(\cdot )$ is the Bessel function of the first kind of order $%
(d-2)/2.$

\begin{definition}
	The random field $\eta (x),$ $x\in \mathbb{R}^{d},$ defined above is said
	to possess an absolutely continuous spectrum if there exists a function $%
	f(\cdot )$ such that 
	\begin{equation*}
	\Phi (z)=2\pi ^{d/2}\Gamma ^{-1}\left( d/2\right) \int_{0}^{z}u^{d-1}f(u)\,%
	\mathrm{d}u,\quad z\geq 0,\quad u^{d-1}f(u)\in L_{1}(\mathbb{R}_{+}).
	\end{equation*}%
	The function $f(\cdot )$ is called the isotropic spectral density function
	of the field~$\eta (x).$ The field~$\eta (x)$ with an absolutely continuous spectrum has the isonormal spectral representation 
	\begin{equation*}
	\eta (x)=\int_{\mathbb{R}^d}e^{i<\lambda ,x>}\sqrt{f\left( \left\| \lambda
		\right\| \right) }W(\mathrm{d}\lambda ),
	\end{equation*}
	where $W(\cdot )$ is the complex Gaussian white noise random measure on $%
	\mathbb{R}^d.$
\end{definition}

Let $U$ and $V$ be two independent and uniformly distributed on the hypersurface $
\partial\Delta (r)$ random vectors. We denote by $\psi _{\Delta (r)}(\rho
),$ $\rho \geq 0,$ the pdf of the distance $\left\| U-V\right\| $ between $U$ and $V.$   Note that $\psi _{\Delta (r)}(\rho
)=0$ if $\rho > diam\left\{ \Delta (r)
\right\}.$
Using the above notations, we obtain the  representation
\[
\int\limits_{\partial\Delta (r)}\int\limits_{\partial\Delta (r)}G(\left\| x-y\right\| )\,d\sigma(x)\,
d\sigma(y)=
\left| \partial\Delta \right| ^{2}r^{2d-2}\mathbf{E}\ G(\left\| U-V\right\| )=
\]
\begin{equation}\label{dint}
=\left| \partial\Delta \right| ^{2}r^{2d-2}\int_{0}^{diam\left\{ \Delta
	(r)\right\} }G(\rho )\ \psi _{\Delta (r)}(\rho )d\rho.
\end{equation}

\begin{rem}\cite{Iv} 
	If $\partial\Delta(r) = \rm{S}_{d-1}(r)$, then 	
	\[\psi _{\Delta (r)}(\rho) = \frac{1}{\sqrt{\pi}}\Gamma\left(\frac{d}{2}\right)\Gamma^{-1}\left(\frac{d-1}{2}\right)r^{1-d}\rho^{d-2}\left(1 - \frac{\rho ^2}{4u^2}\right)^{\frac{d-3}{2}}, \quad 0 < \rho < 2r.\]
\end{rem}

Let $H_{k}(u)$, $k\geq 0$, $u\in \mathbb{R}$, be the Hermite polynomials,
see~\cite{pec}. These polynomials form a complete orthogonal system in the Hilbert space 
\begin{equation*}
{L}_{2}(\mathbb{R},\phi (w)\,dw)=\left\{ G:\int_{\mathbb{R}}G^{2}(w)\phi
(w)\,\mathrm{d}w<\infty \right\} ,\quad \phi (w):=\frac{1}{\sqrt{2\pi }}e^{-%
	\frac{w^{2}}{2}}.
\end{equation*}

An arbitrary function $G(w)\in {L}_{2}(\mathbb{R},\phi(w )\, dw)$ admits the
mean-square convergent expansion 
\begin{equation*}  \label{herm}
G(w)=\sum_{j=0}^{\infty }\frac{C_{j}H_{j}(w) }{j !}, \qquad C_{j }:=\int_{%
	\mathbb{R}}G(w)H_{j }(w)\phi ( w )\,\mathrm{d}w.
\end{equation*}

By Parseval's identity 
\begin{equation*}  \label{par}
\sum_{j=0}^\infty\frac{C_{j}^{2}}{j !} =\int_{\mathbb{R}}G^2(w) \phi ( w )\,%
\mathrm{d}w.
\end{equation*}

\begin{definition} \rm{\cite{Taq1}} Let $G(w)\in {L}_{2}(\mathbb{R},\phi (w)\,dw)$ and
	assume there exists an integer $\kappa \in \mathbb{N}$ such that $C_{j}=0$, for all $%
	0\leq j\leq \kappa -1,$ but $C_{\kappa }\neq 0.$ Then $\kappa $ is called
	the Hermite rank of $G(\cdot )$ and is denoted by $H\mbox{rank}\,G.$
\end{definition}

We investigate the random variables
\[
K_r :=\int\limits_{\partial\Delta(r)}G(\eta (x))d\sigma(x) \quad \mbox{and}
\quad K_{r,\kappa} := \frac{C_\kappa}{k!}\int\limits_{\partial\Delta(r)}H_\kappa (\eta (x))d\sigma(x),
\]
where $C_\kappa$ is a $\kappa$-th coefficient of the Hermite series of the function $G(\cdot)$.

\begin{rem}{\label{rem1}}
	If $(\xi _{1},\ldots ,\xi _{2p})$ is a $2p$-dimensional
	zero-mean Gaussian vector with 
	\begin{equation*}
	\mathbf{E}\xi _{j}\xi _{k}=%
	\begin{cases}
	1, & \mbox{if }k=j, \\ 
	r_{j}, & \mbox{if }k=j+p\ \mbox{and }1\leq j\leq p, \\ 
	0, & \mbox{otherwise,}%
	\end{cases}%
	\end{equation*}%
	then 
	\begin{equation*}
	\mathbf{E}\ \prod_{j=1}^{p}H_{k_{j}}(\xi _{j})H_{m_{j}}(\xi
	_{j+p})=\prod_{j=1}^{p}\delta _{k_{j}}^{m_{j}}\ k_{j}!\ r_{j}^{k_{j}}.
	\end{equation*}
\end{rem}

If $G(w)\in \mathbf{L}_{2}(\mathbb{R}^{p},\phi
(\left\| w\right\| )\, dw)$ and $\mathbf{E}G(\eta(x)) = 0$ then the integral functional $K_r$
can be represented as
\[
K_r=\sum_{j=1}^{\infty }\frac{C_{j}}{j !}
\int\limits_{\partial\Delta (r)}H_{j}(\eta (x))\, d\sigma(x).\]
Therefore $\mathbf{E}K_r=0$ and by Remark~\ref{rem1} the variance is equal
\begin{equation}\label{varg}
\mathbf{Var}\, K_r=
\sum_{j=1}^{\infty }\frac{C_{j }^{2}}{j !}
\int\limits_{\partial\Delta (r)}\int\limits_{\partial\Delta (r)}\mathrm{B}^j(\left\| x-y\right\|
)d\sigma(x)d\sigma(y).
\end{equation}

\begin{definition}
	\rm{\cite{bin}} A measurable function $L:(0,\infty )\rightarrow
	(0,\infty )$ is said to be slowly varying at infinity if for all $t>0,$%
	\begin{equation*}
	\lim\limits_{\lambda \rightarrow \infty }\frac{L(\lambda t)}{L(\lambda )}=1.
	\end{equation*}
\end{definition}

By the representation theorem~\cite[Theorem 1.3.1]{bin}, there exists $C > 0$
such that for all $r \ge C$ the function $L(\cdot)$ can be written in the
form 
\begin{equation*}
L(r) = \exp \left(\zeta_1(r) + \int_C^r \frac{\zeta_2(u)}{u}\,\mathrm{d}u
\right),
\end{equation*}
where $\zeta_1(\cdot)$ and $\zeta_2(\cdot)$ are such measurable and bounded
functions that
$\zeta_2(r)\to 0$ and $\zeta_1(r)\to C_0$ $(|C_0|<\infty),$
when $r\to \infty.$

If $L(\cdot )$ varies slowly, then $r^{a}L(r)\rightarrow \infty ,$ $%
r^{-a}L(r)\rightarrow 0$ for an arbitrary $a>0$ when $r\rightarrow \infty ,$
see Proposition 1.3.6 \cite{bin}.

\begin{definition}
	\cite{bin} A measurable function $g:(0,\infty )\rightarrow (0,\infty )$ is
	said to be regularly varying at infinity, denoted $g(\cdot )\in R_{\tau }$,
	if there exists $\tau $ such that, for all $t>0,$ it holds that 
	\begin{equation*}
	\lim\limits_{\lambda \rightarrow \infty }\frac{g(\lambda t)}{g(\lambda )}%
	=t^{\tau }.
	\end{equation*}
\end{definition}

\begin{definition}
	\label{sr2}\cite{bin} Let	$g:(0,\infty )\rightarrow (0,\infty )$ be a measurable function and $g(x)\rightarrow 0$
	as $x\rightarrow \infty$. A slowly varying function $L(\cdot )$ is said to be slowly varying with
	remainder of type 2, or that it belongs to the class SR2, if 
	\begin{equation*}
	\forall \lambda >1:\quad \frac{L(\lambda x)}{L(x)}-1\sim k(\lambda
	)g(x),\quad x\rightarrow \infty ,
	\end{equation*}%
	for some function $k(\cdot )$.
	
	If there exists $\lambda $ such that $k(\lambda )\neq 0$ and $k(\lambda \mu
	)\neq k(\mu )$ for all $\mu $, then $g(\cdot )\in R_{\tau }$ for some $\tau
	\leq 0$ and $k(\lambda )=ch_{\tau }(\lambda )$, where 
	\begin{equation}
	h_{\tau }(\lambda )=%
	\begin{cases}
	\ln (\lambda ),\quad if\,\tau =0, \\ 
	\frac{\lambda ^{\tau }-1}{\tau },\quad \,if\,\tau \neq 0.%
	\end{cases}
	\label{htau}
	\end{equation}
\end{definition}

\begin{rem}\label{log}
	An example of a function that satisfies Definition~\ref{sr2} for $\tau = 0$ is $L(x) = \ln(x).$ Indeed, 
	\[ 
	\frac{L(\lambda x)}{L(x)} - 1 = \frac{\ln(\lambda) + \ln(x)}{\ln(x)} -1 = \ln(\lambda)\cdot\frac{1}{\ln(x)}.
	\]
\end{rem}

\begin{assumption}
	\label{ass1} Let $\eta (x),$ $x\in \mathbb{R}^{d}$, be a homogeneous
		isotropic Gaussian random field with $\mathbf{E}\eta (x)=0$ and a covariance
		function $B(x)$ such that
	\begin{equation*}
	B(0)=1,\quad B(x)=\mathbf{E}\eta (0) \eta (x)= \left\Vert x\right\Vert
	^{-\alpha }L_{0}(\left\Vert x\right\Vert ),
	\end{equation*}
	where $L_{0}(\left\Vert \cdot\right\Vert )$ is a function slowly varying at
	infinity.
\end{assumption}

\begin{assumption}
	\label{ass2} The random field $\eta(x),$ $x \in \mathbb{R}^d,$ has the
	spectral density
	\begin{equation}\label{f}
	f(\left\| \lambda \right\| )= c_2(d,\alpha )\left\| \lambda \right\|
	^{\alpha -d}L\left( \frac 1{\left\| \lambda \right\| }\right),
	\end{equation}
	where 
	\begin{equation*}
	c_2(d,\alpha ):=\frac{\Gamma \left( \frac{d-\alpha }2\right) }{2^\alpha \pi
		^{d/2}\Gamma \left( \frac \alpha 2\right) },
	\end{equation*}
	and $L(\left\Vert \cdot\right\Vert )$ is a locally bounded function which is
	slowly varying at infinity and satisfies for sufficiently large $r$ the
	condition 
	\begin{equation}  \label{gr}
	\left|1-\frac{L(tr)}{L(r)}\right|\le C\,g(r)h_{\tau}(t),\ t\ge 1,
	\end{equation}
	where $g(\cdot) \in R_{\tau} , \tau \le 0$, such that $g(x) \to 0, \ x \to
	\infty$, and $h_{\tau}(t)$ is defined by~\rm{(\ref{htau}).}
\end{assumption}

\begin{rem}\label{eql}
	By Tauberian and Abelian theorems, see \cite{leoole}, for $L_0(\cdot)$ and $L(\cdot)$ given in  Assumptions~\ref{ass1} and \ref{ass2} it holds  $L_0(r) \sim L(r),$ $r\to +\infty.$  
\end{rem}

\begin{rem}\cite{New}
	~\label{rem0} If $L$ satisfies {\rm{(\ref{gr})}}, then for any $%
	k\in \mathbb{N}$, $\delta >0$, and sufficiently large~$r$ 
	\begin{equation*}
	\left\vert 1-\frac{L^{k/2}(tr)}{L^{k/2}(r)}\right\vert \leq C\,g(r)h_{\tau
	}(t)t^{\delta },\ t\geq 1.
	\end{equation*}
\end{rem}

\begin{definition}
	Let $Y_1$ and $Y_2$ be arbitrary random variables. The uniform (Kolmogorov)
	metric for the distributions of $Y_1$ and $Y_2$ is defined by the formula 
	\begin{equation*}
	{\rho}\left( Y_1,Y_2\right) =\underset{z\in \mathbb{R}}{\sup }\left| P\left(
	Y_1\leq z\right) -P\left( Y_2\leq z\right) \right| .
	\end{equation*}
\end{definition}

The next result follows from Lemma~1.8~\cite{pet}.

\begin{lemma}
	\label{lem1} If $X,Y$ and $Z$ are arbitrary random variables, then for any $%
	\varepsilon >0$ 
	\begin{equation*}
	\rho \left( X+Y,Z\right) \leq {\rho }(X,Z)+\rho \left( Z+\varepsilon
	,Z\right) +P\left( \left\vert Y\right\vert \geq \varepsilon \right) .
	\end{equation*}
\end{lemma}

\section{Results on the asymptotic behavior}\label{abs}

In this section we are interested in the asymptotic distribution of the random variable $K_r =\int\limits_{\partial\Delta(r)}G(\eta (x))d\sigma(x).$ First, we prove Theorem~\ref{th4}, which is an analogue of the so called reduction theorem, see Theorem~4 in \cite{Souj}, in the case of hypersurface integrals. Using this result, in Theorem~\ref{th5} we derive normalizing coefficients and limit distributions of the random variable $K_r$ that depend on the Hermite rank $\kappa$ of the function $G(\cdot)$. 

\begin{theorem}\label{th4}
Suppose that $H\mathrm{rank}\,G=\kappa \in \mathbb{N}$ and $\eta (x),$ $x\in \mathbb{R}^{d},$
satisfies Assumption~{\rm\ref{ass1}} for $\alpha\in(0, (d-1)/\kappa)$. If at least one of the following random variables
\begin{equation*}
\frac{K_{r}}{\sqrt{\mathbf{Var}\text{ }K_{r}}},\quad \frac{K_{r}}{\sqrt{%
		\mathbf{Var}\ K_{r,\kappa }}}\quad \mbox{and}\quad \frac{K_{r,\kappa }}{%
	\sqrt{\mathbf{Var}\ K_{r,\kappa }}},
\end{equation*}%
has a limit distribution, then the limit distributions of the other random variables also exist and
they coincide when $r\rightarrow \infty .$
\end{theorem}

\begin{proof} Let
	\[
	V_r:=\sum_{j\geq \kappa+1}\frac{C_j}{j !}
	\int\limits_{\partial\Delta(r)}H_j (\eta (x))d\sigma(x),
	\]
	then by Remark~\ref{rem1}
	\[
	\mathbf{Var}\, K_r= \mathbf{Var}\, K_{r,\kappa}+ \mathbf{Var}\, V_r.
	\]
	By (\ref{dint}) and (\ref{varg})
	\begin{eqnarray*}
		\mathbf{Var} \, K_{r,\kappa}&=& \frac{C_\kappa ^2}{\kappa !}
		\int\limits_{\partial\Delta(r)}\int\limits_{\partial\Delta(r)}\left\|
		x-y\right\|^{-\alpha\kappa} L_0^\kappa\left(\left\| x-y\right\| \right)\, \mathrm{d}\sigma(x)\, \mathrm{d}\sigma(y) \\
		&=&|\partial\Delta|^2 r^{2d-2-\alpha\kappa}\frac{C_\kappa ^2}{\kappa !}
		\int\limits_0^{diam\left\{ \Delta
			\right\}} z^{-\alpha\kappa} L_0^\kappa\left(rz\right)\psi _{\Delta}(z)dz.\end{eqnarray*}
	
	If $\alpha \in (0,(d-1)/\kappa)$ then by asymptotic properties of integrals of slowly varying functions (see Theorem~2.7 \cite{sen}) we get
	
	\begin{eqnarray*}
		\mathbf{Var} \, K_{r,\kappa}&=&c_1(\kappa,\alpha,\Delta)\, |\partial\Delta|^2 \frac{C_\kappa ^2}{\kappa !}\,
		r^{2d-2-\kappa\alpha}L_0^\kappa(r)(1+o(1)),
		\quad r\to \infty,
	\end{eqnarray*}
	where
	\[c_1(\kappa,\alpha,\Delta):=\int\limits_0^{diam\left\{ \Delta
		\right\}} z^{-\alpha\kappa}\psi _{\Delta }(z)dz.\]

	Similar to $\mathbf{Var} \ K_{r,\kappa}$ we obtain
	\[\mathbf{Var}\, V_r = |\partial\Delta|^2r^{2d-2}\sum_{j\geq \kappa+1}\frac{C_j ^2}{j!}
	\int\limits_0^{r\cdot diam\left\{ \Delta
		\right\}} z^{-\alpha j} L_0^j\left(z\right)\psi _{\Delta (r)}(z)dz.
	\]
	It follows from  $z^{-\alpha} L\left(z\right)\in[0,1],$ $z\ge 0,$  that
	\begin{eqnarray*}
		\mathbf{Var}\, V_r &\leq& |\partial\Delta|^2r^{2d-2-(\kappa+1)\alpha}\sum_{j\geq \kappa+1}\frac{C_j ^2}{j !}
		\int\limits_0^{diam\left\{ \Delta
			\right\}} z^{-\alpha(\kappa+1)} L_0^{\kappa+1}\left(rz\right)\psi _{\Delta}(z)dz\\
		&=&  |\partial\Delta|^2r^{2d-2-\kappa\alpha}L_0^{\kappa}(r)\sum_{j\geq \kappa+1}\frac{C_j ^2}{j !}
		\int\limits_0^{diam\left\{ \Delta
			\right\}} z^{-\alpha\kappa} \frac{L_0^{\kappa}\left(rz\right)}{L_0^{\kappa}(r)}  \frac{L_0\left(rz\right)}{(rz)^{\alpha}}\psi _{\Delta}(z)dz.  \end{eqnarray*}
	
	Let us split the above integral into two parts $I_1$ and $I_2$ with the ranges of integration $[0,r^{-\beta}]$ and $(r^{-\beta},diam\left\{ \Delta
	\right\}]$ respectively, where $\beta\in(0,1).$
	
	As $z^{-\alpha} L_0\left(z\right)\in[0,1],$ $z\ge 0,$ we can estimate the first integral as follows
	\[I_1\le\int\limits_0^{r^{-\beta}} z^{-\alpha\kappa} \frac{L_0^{\kappa}\left(rz\right)}{L_0^{\kappa}(r)}  \psi _{\Delta}(z)dz\le
	\frac{\sup_{0\le s\le r^{1-\beta}}s^{\delta}L_0^{\kappa}\left(s\right)}{r^{\delta}L_0^{\kappa}(r)} \int\limits_0^{r^{-\beta}}z^{-(\delta+\alpha\kappa)}  \psi _{\Delta}(z)dz\]
	\begin{equation}\label{int1}
	\le\left(\frac{\sup_{0\le s\le r}s^{\delta/k}L_0\left(s\right)}{r^{\delta/k}L_0(r)}\right)^{\kappa} \int\limits_0^{r^{-\beta}} z^{-(\delta+\alpha\kappa)}  \psi _{\Delta}(z)dz.
	\end{equation}

	By Theorem~1.5.3 \cite{bin} and the definition of slowly varying functions
	\[\lim_{r\to\infty}\frac{\sup_{0\le s\le r}s^{\delta/k}L_0\left(s\right)}{r^{\delta/k}L_0(r)}=1.\]
	
	By (\ref{dint}) we can rewrite the integral in (\ref{int1}) as follows
	\[\int\limits_0^{r^{-\beta}} z^{-(\delta+\alpha\kappa)}  \psi _{\Delta}(z)dz=\left| \partial\Delta \right| ^{-2}\int\limits_{\partial\Delta}\int\limits_{\partial\Delta}\chi(\left\| x-y\right\|\le r^{-\beta})\left\| x-y\right\|^{-(\delta+\alpha\kappa)}\,\mathrm{d}\sigma(x)\,\mathrm{d}\sigma(y)\]
	\[
	\leq \left| \partial\Delta \right| ^{-2}\int\limits_{\partial\Delta}\max\limits_{y}\left(\int\limits_{\partial\Delta}\chi(\left\| x-y\right\|\le r^{-\beta})\left\| x-y\right\|^{-(\delta+\alpha\kappa)}\,\mathrm{d}\sigma(x)\right)\mathrm{d}\sigma(y)
	\]
	\[
	= \left| \partial\Delta \right| ^{-1}\max\limits_{y}\left(\int\limits_{\partial\Delta}\chi(\left\| x-y\right\|\le r^{-\beta})\left\| x-y\right\|^{-(\delta+\alpha\kappa)}\,\mathrm{d}\sigma(x)\right)
	\]
	\[
	= \left| \partial\Delta \right| ^{-1}\max\limits_{y}\left(\int\limits_{\partial\Delta\cap B(y, r^{-\beta})}\left\| x-y\right\|^{-(\delta+\alpha\kappa)}\,\mathrm{d}\sigma(x)\right).
	\]
	
	Since $\partial\Delta$ is Ahlfors-David regular, applying upper-bound from (\ref{adr}) we get
	\[\int\limits_{\partial\Delta\cap B(y, r^{-\beta})}\left\| x-y\right\|^{-(\delta+\alpha\kappa)}\,\mathrm{d}\sigma(x) = \sum\limits_{i=0}^{\infty}\int\limits_{\partial\Delta\cap\left[ B(y, r^{-\beta}2^{-i})\backslash B(y, r^{-\beta}2^{-i-1}\right]}\left\| x-y\right\|^{-(\delta+\alpha\kappa)}\,\mathrm{d}\sigma(x)
	\]
	\[
	\leq \sum\limits_{i=0}^{\infty}\int\limits_{\partial\Delta\cap\left[ B(y, r^{-\beta}2^{-i})\backslash B(y, r^{-\beta}2^{-i-1}\right]}r^{\beta(\delta+\alpha\kappa)}2^{(i+1)(\delta+\alpha\kappa)}\,\mathrm{d}\sigma(x) \leq 
	\sum\limits_{i=0}^{\infty}r^{\beta(\delta+\alpha\kappa)}2^{(i+1)(\delta+\alpha\kappa)}\]
	\[\times\int\limits_{\partial\Delta\cap B(y, r^{-\beta}2^{-i})}\mathrm{d}\sigma(x) \leq Cr^{\beta(\delta+\alpha\kappa)}\sum\limits_{i=0}^{\infty}2^{(i+1)(\delta+\alpha\kappa)}r^{-\beta(d-1)}2^{-i(d-1)}
	\]
	\[
	= \frac{C2^{\delta + \alpha\kappa}}{1 - 2^{-(d - (1 +\delta+\alpha\kappa))}}r^{-\beta(d - (1 +\delta+\alpha\kappa))}.
	\]
	Thus, we have
	\begin{equation}\label{I11}
	\int\limits_0^{r^{-\beta}} z^{-(\delta+\alpha\kappa)}  \psi _{\Delta}(z)dz \leq Cr^{-\beta(d-(1+\delta+\alpha\kappa))}.
	\end{equation}
	
	For the second integral we obtain
	\[I_2\le
	\frac{\sup_{r^{1-\beta}\le s\le r\cdot diam\left\{ \Delta
			\right\}}s^{\delta}L_0^{\kappa}\left(s\right)}{r^{\delta}L_0^{\kappa}(r)}\cdot \sup_{r^{1-\beta}\le s\le r\cdot diam\left\{ \Delta
		\right\}}\frac{L_0\left(s\right)}{s^\alpha}\cdot \int\limits_0^{diam\left\{ \Delta \right\}} z^{-(\delta+\alpha\kappa)}  \psi _{\Delta}(z)dz.\]
	
	Using Theorem~1.5.3 \cite{bin} we conclude that
	\[\lim_{r\to\infty}\frac{\sup_{r^{1-\beta}\le s\le r\cdot diam\left\{ \Delta
			\right\}}s^{\delta}L_0^{\kappa}\left(s\right)}{r^{\delta}L_0^{\kappa}(r)}\le\lim_{r\to\infty}\frac{\sup_{0\le s\le r\cdot diam\left\{ \Delta
			\right\}}s^{\delta}L_0^{\kappa}\left(s\right)}{(r\cdot diam\left\{ \Delta
		\right\})^{\delta}L_0^{\kappa}(r\cdot diam\left\{ \Delta \right\})}\]
	\[\times\lim_{r\to\infty}\frac{diam^{\delta}\left\{ \Delta
		\right\}L_0^{\kappa}(r\cdot diam\left\{ \Delta \right\})}{L_0^{\kappa}(r)}=diam^{\delta}\left\{ \Delta \right\}.\]
	
	By Proposition~1.3.6 and Theorem~1.5.3 \cite{bin} it follows that
	\[\sup\limits_{r^{1-\beta}\le s\le r\cdot diam\left\{ \Delta\right\}}\frac{L_0\left(s\right)}{s^\alpha}\le \frac{\sup_{s\ge r^{1-\beta}}s^{-\alpha}L_0\left(s\right)}{r^{-\alpha(1-\beta)}L_0\left(r^{1-\beta}\right)}\cdot\frac{L_0\left(r^{1-\beta}\right)}{r^{\delta(1-\beta)}} \cdot r^{(\delta-\alpha)(1-\beta)}\]
	\begin{equation}\label{I21}
	=o(r^{(\delta-\alpha)(1-\beta)}).
	\end{equation}
	We can choose $\beta=1/2$ and make $\delta$ arbitrary close to 0. Then by (\ref{I11}), (\ref{I21}) we obtain
	\[
	\lim_{r\to \infty }\frac{\mathbf{Var}\, V_r}{\mathbf{Var}\, K_{r}}=0\quad \mbox{and}\quad
	\lim_{r\to \infty }\frac{\mathbf{Var}\, K_{r}}{\mathbf{Var}\, K_{r,\kappa}} = 1.
	\]
	
	Thus
	\[
	\lim_{r\to \infty }\, \mathbf{E}\left(\frac{K_{r}}{\sqrt{ \mathbf{Var} \ K_{r}}}-\frac{K_{r,\kappa}}{\sqrt{ \mathbf{Var} \ K_{r,\kappa}}}\right)^2=\lim_{r\to \infty }\frac{\mathbf{E}\left(V_r+\left(1-\sqrt{\frac{\mathbf{Var} K_{r}}{\mathbf{Var} K_{r,\kappa}}}\right)K_{r,\kappa}\right)^2}{\mathbf{Var} K_{r}} =0,\]
	and
	\[
	\lim_{r\to \infty }\, \mathbf{E}\left(\frac{K_{r}}{\sqrt{ \mathbf{Var} \ K_{r,\kappa}}}-\frac{K_{r,\kappa}}{\sqrt{ \mathbf{Var} \ K_{r,\kappa}}}\right)^2=\lim_{r\to \infty }\frac{\mathbf{E}\left(V_r\sqrt{\frac{\mathbf{Var} K_{r}}{\mathbf{Var} K_{r,\kappa}}}\right)^2}{\mathbf{Var} K_{r}} =0\]
	which completes the proof.
\end{proof}

\begin{lemma}\label{finint}
If $\tau_1,...,\tau_\kappa,$ $\kappa\ge 1,$ are such positive constants, that $\sum_{i=1}^\kappa \tau_i <d-1,$  then
	\begin{equation}\label{finv}
	\int\limits_{\mathbb{R}^{d\kappa}}|\mathcal{K}(\lambda _1+\cdots
	+\lambda _\kappa)|^2 \frac{\mathrm{d}\lambda _1\ldots \mathrm{d}\lambda _\kappa}{\left\| \lambda
		_1\right\| ^{d-\tau_1}\cdots \left\| \lambda _\kappa\right\| ^{d-\tau_\kappa}}<\infty .
	\end{equation}
\end{lemma}

\begin{proof}
	For $\kappa= 1$ we get $d-\tau_1>1$. Using integration formula for polar coordinates, and the fact that $|\mathcal{K}(\lambda)| 
	\leq |\partial\Delta|$ for all $\lambda \in \mathbb{R}^d$ we get	
	\[
	\int\limits_{\mathbb{R}^{d}}|\mathcal{K}(\lambda)|^2  \frac{\mathrm{d}\lambda}{\left\| \lambda\right\| ^{d-\tau_1}}= \int\limits_{0}^{\infty}r^{d-1}\int\limits_{S_{d-1}(1)}\frac{|\mathcal{K}(\omega r)|^2}{r^{d-\tau_1}}\,\mathrm{d}\omega\mathrm{d}r = \int\limits_{0}^{r_0}r^{d-1}\int\limits_{S_{d-1}(1)}\frac{|\mathcal{K}(\omega r)|^2}{r^{d-\tau_1}}\,\mathrm{d}\omega\mathrm{d}r \]
	\[+ \int\limits_{r_0}^{\infty}r^{d-1}\int\limits_{S_{d-1}(1)}\frac{|\mathcal{K}(\omega r)|^2}{r^{d-\tau_1}}\,\mathrm{d}\omega\mathrm{d}r
	\leq |\partial\Delta|^2\int\limits_{0}^{r_0}\frac{r^{d-1}\mathrm{d}r}{r^{d-\tau_1}} + \int\limits_{r_0}^{\infty}r^{d-1}\int\limits_{S_{d-1}(1)}\frac{|\mathcal{K}(\omega r)|^2}{r^{d-\tau_1}}\,\mathrm{d}\omega\mathrm{d}r.\]
	By (\ref{l2a}) we obtain
	\[
	\int\limits_{\mathbb{R}^{d}}|\mathcal{K}(\lambda)|^2  \frac{\mathrm{d}\lambda}{\left\| \lambda\right\| ^{d-\tau_1}} \leq |\partial\Delta|^2\int\limits_{0}^{r_0}\frac{\mathrm{d}r}{r^{1-\tau_1}} +C\int\limits_{r_0}^{\infty}\frac{r^{-d+1}}{r^{1-\tau_1}}\mathrm{d}r
	= |\partial\Delta|^2\int\limits_{0}^{r_0}\frac{\mathrm{d}r}{r^{1-\tau_1}} +C\int\limits_{r_0}^{\infty}\frac{\mathrm{d}r}{r^{d-\tau_1}} <\infty.
	\]
	
	For $\kappa> 1$ we can obtain (\ref{finv}) by the recursive estimation routine and the change of variables $\tilde{\lambda}_{\kappa-1}={\lambda_{\kappa-1}}/{\left\| u\right\|}:$	
	\[
	\int\limits_{\mathbb{R}^{d\kappa}}  \frac{|\mathcal{K}(\lambda _1+\cdots
		+\lambda _{\kappa})|^2\mathrm{d}\lambda _1\ldots \mathrm{d}\lambda _\kappa}{\left\| \lambda
		_1\right\| ^{d-\tau_1}\cdots \left\| \lambda _\kappa\right\| ^{d-\tau_\kappa}}=|u=\lambda_{\kappa-1}+\lambda_\kappa|=
	\int\limits_{\mathbb{R}^{d(\kappa-1)}}|\mathcal{K}(\lambda _1+\cdots
	+\lambda _{\kappa-2}+u)|^2 \]
	\[\times \int\limits_{\mathbb{R}^{d}} \frac{\mathrm{d}\lambda _{\kappa-1}}{\left\| \lambda _{\kappa-1}\right\| ^{d-\tau_{\kappa-1} }\left\|u- \lambda
		_{\kappa-1}\right\| ^{d-\tau_{\kappa}}}\cdot \frac{\mathrm{d}\lambda _1\ldots \mathrm{d}\lambda _{\kappa-2}\, \mathrm{d}u}{\left\| \lambda
		_1\right\| ^{d-\tau_{1} }\cdots \left\| \lambda _{\kappa-2}\right\| ^{d-\tau_{\kappa-2} }}
	\]
	\[=
	\int\limits_{\mathbb{R}^{d(\kappa-1)}}\frac{|\mathcal{K}(\lambda _1+\cdots
		+\lambda _{\kappa-2}+u)|^2  \mathrm{d}\lambda _1\ldots \mathrm{d}\lambda _{\kappa-2}}{\left\| \lambda
		_1\right\| ^{d-\tau_{1} }\cdots \left\| \lambda _{\kappa-2}\right\|^{d-\tau_{\kappa-2} }\left\| u\right\|^{d-\tau_{\kappa-1}-\tau_{\kappa} }} \int\limits_{\mathbb{R}^{d}} \frac{\mathrm{d}\tilde{\lambda}_{\kappa-1}\mathrm{d}u}{\left\|\tilde{\lambda}_{\kappa-1}\right\| ^{d-\tau_{\kappa-1} }\left\|\frac{u}{\left\| u\right\|}- \tilde{\lambda}
		_{\kappa-1}\right\| ^{d-\tau_{\kappa}}}\, \]
	\[\le C \int\limits_{\mathbb{R}^{d(\kappa-1)}}|\mathcal{K}(\lambda _1+\cdots
	+\lambda _{\kappa-2}+u)|^2   \frac{\mathrm{d}\lambda _1\ldots \mathrm{d}\lambda _{\kappa-2}\, \mathrm{d}u}{\left\| \lambda
		_1\right\| ^{d-\tau_{1} }\cdots \left\| \lambda _{\kappa-2}\right\|^{d-\tau_{\kappa-2} }\left\| u\right\|^{d-\tau_{\kappa-1}-\tau_{\kappa} }}\]
	\[\le ...\le C \int\limits_{\mathbb{R}^{d}}|\mathcal{K}(u)|^2   \frac{\mathrm{d}u}{\left\| u\right\|^{d-\sum_{i=1}^\kappa \tau_i}}<\infty.\qedhere\]
\end{proof}

\begin{theorem}\label{th5} Let $\eta(x),$ $x\in \mathbb{R}^d,$ be a homogeneous isotropic Gaussian random
field with $\mathbf{E}\eta(x)=0.$  If Assumptions~{\rm{\ref{ass1}}} and {\rm{\ref{ass2}}} hold, $\alpha \in (0,(d-1)/\kappa),$ and $H\mathrm{rank}\,G=\kappa\in\mathbb{N},$  then for $r\to \infty$ the random variable
\[X_{\kappa,r}(\Delta):=r^{(\kappa\alpha )/2-d+1}L^{-\kappa/2}(r)\int\limits_{\partial\Delta(r)}H_\kappa(\eta (x))\,\mathrm{d}\sigma(x)\]
converge weakly to
\begin{equation}\label{Xk}
X_\kappa(\Delta) :=c_2^{\kappa/2}(d,\alpha )  {\int\limits_{\mathbb{R}^{d\kappa}}}^{\prime}\mathcal{K}(\lambda _1+\cdots
+\lambda _\kappa) \frac{W(\mathrm{d}\lambda _1)\ldots W(\mathrm{d}\lambda _\kappa)}{\left\| \lambda
_1\right\| ^{(d-\alpha )/2}\cdots \left\| \lambda _\kappa\right\| ^{(d-\alpha )/2}},\end{equation}
where ${\int\limits_{\mathbb{R}^{d\kappa}}}^{\hspace{-0.03in}\prime}$ denotes the multiple stochastic Wiener-It\^{o} integral.
\end{theorem}

\begin{rem}
	Note, that from the following proof it is clear that it is sufficient to use only (\ref{f}) instead of Assumption~\ref{ass2}.
\end{rem}

\begin{proof}
	Using It\'{o} formula (2.3.1) in \cite{LeoLT} we obtain
	\[
	\int\limits_{\partial\Delta(r) } H_\kappa(\eta (x))\mathrm{d}\sigma(x)=\int\limits_{\partial\Delta(r) }{\int\limits_{\mathbb{R}^{d\kappa}}}^{\prime}e^{i<\lambda _1+\cdots +\lambda
		_\kappa,x>} \prod\limits_{j=1}^\kappa\sqrt{f(\|\lambda _j\|)} W(\mathrm{d}\lambda
	_1)\ldots W(\mathrm{d}\lambda _\kappa)\mathrm{d}\sigma(x).
	\]
	As $\prod\limits_{j=1}^\kappa\sqrt{f(\|\lambda _j\|)}\in L_2(\mathbb{R}^{d\kappa})$ then a stochastic Fubini theorem, see Theorem~5.13.1 in \cite{pec}, can be used to interchange the integrals which results in
	\begin{equation}\label{spl}
		X_{\kappa,r}(\Delta)\stackrel{\mathcal{D}}{=}
		c_2^{\kappa/2}(d,\alpha){\int\limits_{\mathbb{R}^{d\kappa}}}^{\prime}\frac{\mathcal{K}(\lambda _1+\cdots +\lambda _\kappa)Q_r(\lambda _1,\ldots ,\lambda _\kappa)W(\mathrm{d}\lambda
			_1)\ldots W(\mathrm{d}\lambda _\kappa)}{\left\| \lambda _1\right\| ^{(d-\alpha )/2}\cdots
			\left\| \lambda _\kappa\right\| ^{(d-\alpha )/2}},
	\end{equation}
	where
	\begin{equation}\label{qqq}
	Q_r(\lambda _1,\ldots ,\lambda _\kappa): =r^{\kappa(\alpha-d)/2}L^{-\kappa/2}(r)\
	c_2^{-\kappa/2}(d,\alpha)  \left[ \prod\limits_{j=1}^\kappa\left\| \lambda _j\right\| ^{d-\alpha}f\left( \frac{\left\| \lambda _j\right\| }r\right) \right] ^{1/2}.
	\end{equation}

	By the isometry property of multiple stochastic integrals
	\[
	R_r:=\frac{\mathbb{E}\left| X_{\kappa,r}(\Delta)-X_\kappa(\Delta)\right|^2}{c_2^{\kappa}(d,\alpha)}
	=\int\limits_{\mathbb{R}^{d\kappa}}\frac{|\mathcal{K}(\lambda _1+\cdots +\lambda _\kappa)|^2 \left(Q_r(\lambda _1,\ldots ,\lambda _\kappa)-1\right)^2}{\left\| \lambda _1\right\| ^{d-\alpha}\cdots
		\left\| \lambda _\kappa\right\| ^{d-\alpha}}\, \mathrm{d}\lambda
	_1\ldots \mathrm{d}\lambda _\kappa.
	\]
	
	Using (\ref{f}) and properties of slowly varying functions we conclude that $Q_r(\lambda _1,\ldots ,\lambda
	_\kappa)$ converges pointwise to 1, when $r\to \infty.$
	Hence, by Lebesgue's dominated convergence theorem the integral converges to zero if there is some integrable function which dominates integrands for all $r.$
	
	Let us split $\mathbb{R}^{d\kappa}$ into the regions
	\[B_\mu :=\{(\lambda _1,...,\lambda _\kappa)\in\mathbb{R}^{d\kappa}: ||\lambda_j||\le 1,\ \mbox{if}\ \mu_j=-1,\ \mbox{and}\ ||\lambda_j||> 1, \ \mbox{if}\ \mu_j=1, j=1,...,\kappa\}, \]
	where $\mu=(\mu_1,...,\mu_\kappa)\in \{-1,1\}^\kappa$ is a binary vector of length $\kappa.$
	Then we can represent the integral $R_r$ as
	\[R_r:=\bigcup_{\mu\in \{-1,1\}^\kappa}\int\limits_{B_\mu}|\mathcal{K}(\lambda _1+\cdots +\lambda _\kappa)|^2 \left(Q_r(\lambda _1,\ldots ,\lambda _\kappa)-1\right)^2\frac{\mathrm{d}\lambda
		_1\ldots \mathrm{d}\lambda _\kappa}{\left\| \lambda _1\right\| ^{d-\alpha}\cdots
		\left\| \lambda _\kappa\right\| ^{d-\alpha}}.\]
	
	If $(\lambda _1,...,\lambda _\kappa)\in B_\mu$ we estimate the integrand as follows	
	\[\frac{|\mathcal{K}(\lambda _1+\cdots +\lambda _\kappa)|^2 \left(Q_r(\lambda _1,\ldots ,\lambda _\kappa)-1\right)^2}{\left\| \lambda _1\right\| ^{d-\alpha}\cdots
		\left\| \lambda _\kappa\right\| ^{d-\alpha}}\le \frac{2\, |\mathcal{K}(\lambda _1+\cdots +\lambda _\kappa)|^2 }{\left\| \lambda _1\right\| ^{d-\alpha}\cdots
		\left\| \lambda _\kappa\right\| ^{d-\alpha}}\left(Q^2_r(\lambda _1,\ldots ,\lambda _\kappa)+1\right)\]
	\[= \frac{2\, |\mathcal{K}(\lambda _1+\cdots +\lambda _\kappa)|^2 }{\left\| \lambda _1\right\| ^{d-\alpha}\cdots
		\left\| \lambda _\kappa\right\| ^{d-\alpha}}\left(\prod\limits_{j=1}^\kappa ||\lambda _j||^{\mu_j\delta}\cdot \prod\limits_{j=1}^\kappa \frac{\left(\frac{r}{||\lambda _j||}\right)^{\mu_j\delta}L\left(\frac{r}{||\lambda _j||}\right)}{r^{\mu_j\delta}L(r)} +1\right)\]
	\[\le \frac{2\,|\mathcal{K}(\lambda _1+\cdots +\lambda _\kappa)|^2 }{\left\| \lambda _1\right\| ^{d-\alpha}\cdots
		\left\| \lambda _\kappa\right\| ^{d-\alpha}}\left(1+
	\prod\limits_{j=1}^\kappa\left\| \lambda _1\right\| ^{\mu_j\delta}\cdot \sup_{(\lambda _1,...,\lambda _\kappa)\in B_\mu }\prod\limits_{j=1}^\kappa \frac{\left(\frac{r}{||\lambda _j||}\right)^{\mu_j\delta}L\left(\frac{r}{||\lambda _j||}\right)}{r^{\mu_j\delta}L(r)}\right), \]
	where $\delta$ is an arbitrary positive number.
	By Theorem~1.5.3 \cite{bin}
	\[\lim_{r\to \infty}\frac{\sup_{||\lambda _j||\le 1}\left(\frac{r}{||\lambda _j||}\right)^{-\delta}L\left(\frac{r}{||\lambda _j||}\right)}{r^{-\delta}L(r)}=\lim_{r\to \infty}\frac{\sup_{z\ge r}z^{-\delta}L\left(z\right)}{r^{-\delta}L(r)}=1;\]
	
	\[\lim_{r\to \infty}\frac{\sup_{||\lambda _j||> 1}\left(\frac{r}{||\lambda _j||}\right)^{\delta}L\left(\frac{r}{||\lambda _j||}\right)}{r^{\delta}L(r)}=\lim_{r\to \infty}\frac{\sup_{z\in[0,r]}z^{\delta}L\left(z\right)}{r^{\delta}L(r)}=1.\]
	
	Therefore, there exists $r_0>0$ such that for all $r\ge r_0$ and  $(\lambda _1,...,\lambda _\kappa)\in B_\mu$
	\[\frac{|\mathcal{K}(\lambda _1+\cdots +\lambda _\kappa)|^2 \left(Q_r(\lambda _1,\ldots ,\lambda _\kappa)-1\right)^2}{\left\| \lambda _1\right\| ^{d-\alpha}\cdots
		\left\| \lambda _\kappa\right\| ^{d-\alpha}}\le \frac{2\, |\mathcal{K}(\lambda _1+\cdots +\lambda _\kappa)|^2 }{\left\| \lambda _1\right\| ^{d-\alpha}\cdots
		\left\| \lambda _\kappa\right\| ^{d-\alpha}}\]
	\begin{equation}\label{upper}
	+2\,C\,\frac{|\mathcal{K}(\lambda _1+\cdots +\lambda _\kappa)|^2 }{\left\| \lambda _1\right\| ^{d-\alpha-\mu_1\delta}\cdots
		\left\| \lambda _\kappa\right\| ^{d-\alpha-\mu_\kappa\delta}}.
	\end{equation}
	
	By Lemma~\ref{finint}, if we chose $\delta\in \left(0,\min \left(\alpha,{(d-1)}/{\kappa}-\alpha\right)\right),$ the upper bound in (\ref{upper}) is an integrable function on each  $B_\mu$ and hence on $\mathbb{R}^{d\kappa}$ too.
	By Lebesgue's dominated convergence theorem $\lim_{r\to \infty} \mathbf{E}\left| X_{\kappa,r}(\Delta)-X_\kappa(\Delta)\right|^2=0,$ which completes the proof.
\end{proof}

 \section{Application to sojourn measures}\label{sojs}   
    An important example of Theorem~\ref{th5} is sojourn measures of random fields defined on hypersurfaces, see \cite{Adl}, \cite{LeoNew}. Namely, consider an application of Theorem~\ref{th5} to the functionals
    \[\int\limits_{\partial \Delta(r)}\chi(S(\eta(x)) > b)\mathrm{d}\sigma(x),\]
    where $S:\mathbb{R} \rightarrow \mathbb{R}$ is a such function that the set $\{t: S(t) >b\}$ can be represented as a finite union of intervals $(t_1, t_2), \, -\infty\leq t_1 <t_2\leq+\infty.$ Examples of the function $S(\cdot)$ are polynomials or other smooth functions having finite number of zeros.
    \begin{rem}
    	As particular cases, this construction includes $\int\limits_{\rm{S}_{d-1}(r)}\chi(\eta(x) > b)\mathrm{d}\sigma(x)$ and $\int\limits_{\rm{S}_{d-1}(r)}\chi(|\eta(x)| > b)\mathrm{d}\sigma(x)$ considered in \cite{Iv}.
    \end{rem}
    
     As for some $ N \geq 1$ it holds $\{t: S(t)>b\} = \bigcup\limits_{i=1}^{N}(t_i, t_{i+1}),$ where the intervals $(t_i, t_{i+1})$ are  disjoint, we have to study    
    \[\int\limits_{\partial \Delta(r)}\chi\left(\eta(x) \in \bigcup\limits_{i=1}^{N}(t_i, t_{i+1})\right)\mathrm{d}\sigma(x) = \sum\limits_{i = 1}^{N}\int\limits_{\partial \Delta(r)}\chi\left(\eta(x) \in (t_i,t_{i+1})\right)\mathrm{d}\sigma(x).
    \]
    
    Note, that the indicator function $\chi(\omega > t)$ can be expanded in the Hermite series as 
    \[
    \chi(\omega > t) = \sum\limits_{j = 0}^{\infty}\frac{C_j^{(t)}H_j(\omega)}{j!},
    \]
    where \[C_j^{(t)} = 
    \begin{cases}
    1 - \Phi(t), & j=0,\\
    \phi(t)H_{j-1}(t), & j\geq 1,
    \end{cases}\]
    and $\Phi(\cdot)$ and $\phi(\cdot)$ are the cdf and pdf for $\mathcal{N}(0,1)$ respectively.
    
    Then,
    \[
    \chi\left(\omega \in (t_i,t_{i+1})\right) = \chi\left(\omega > t_i\right) - \chi\left(\omega > t_{i+1}\right) = \Phi(t_{i+1}) - \Phi(t_i)
    \]
    
    \[
    + \sum\limits_{j = 1}^{\infty}\frac{\phi(t_i)H_{j-1}(t_i) - \phi(t_{i+1})H_{j-1}(t_{i+1})}{j!}H_j(w),
    \]
    where $\phi(\pm \infty) = 0.$
    
    Hence, 
    \[
    \sum\limits_{i = 1}^{N}\chi(\omega \in (t_i, t_{i+1})) = \sum\limits_{i = 1}^{N}\left(\Phi(t_{i+1}) - \Phi(t_i)\right)
    \]
    \[+
    \sum\limits_{j = 1}^{\infty}\sum\limits_{i = 1}^{N}\frac{\phi(t_i)H_{j-1}(t_i) - \phi(t_{i+1})H_{j-1}(t_{i+1})}{j!}H_j(w).
    \]
    
    Therefore, the Hermite rank of the function $\chi(S(x) > b)$ is such $j^{*}\geq 1$ that it is the smallest number for which
    \begin{equation*}
    C_{j^{*}\mkern-8mu, b} = \sum\limits_{i = 1}^{N}\phi(t_i)H_{j^{*}-1}(t_i) - \phi(t_{i+1})H_{j^{*}-1}(t_{i+1}) \neq 0.
    \end{equation*}
    
\begin{theorem}
	Let $j^{*} = \min\{j \in \mathbb{N}: \sum\limits_{i = 1}^{N}\phi(t_i)H_{j-1}(t_i) - \phi(t_{i+1})H_{j-1}(t_{i+1}) \neq 0\}.$ Then, under assumptions of Theorem~\ref{th5}
	\[
	X_{\kappa, r}(\Delta) = r^{(\kappa\alpha )/2-d+1}L^{-\kappa/2}(r)\int\limits_{\partial\Delta(r)}\chi(S(\eta (x)) > b)\,\mathrm{d}\sigma(x)
	\]
	converges to $\frac{C_{\kappa, b}}{\kappa!}X_\kappa(\Delta),$ where  $X_{\kappa}(\Delta)$ is given by (\ref{Xk}), and $\kappa = j^{*}$.
\end{theorem}

\begin{example}
	Let us study $\int\limits_{\partial\Delta(r)}\chi(\eta^l (x) > b)\,\mathrm{d}\sigma(x).$ If $l$ is odd, then $\chi\left(\omega^l > b\right) = \chi\left(\omega > b^{1/l}\right).$ In this case $C_{1,b} =\phi(b^{1/l}) \neq 0$ and the asymptotic is given by $\phi(b^{1/l})X_1(\Delta)$ which has a Gaussian distribution.
	
	If $l$ is even, then for $b > 0$ it holds $\chi\left(\omega^l > b\right) = \chi\left(\omega > b^{1/l}\right) + \chi\left(\omega < - b^{1/l}\right) = 1- \chi\left(-b^{1/l}<\omega<b^{1/l}\right).$
	In this case, $C_{1,b} = \phi(-b^{1/l}) - \phi(b^{1/l}) = 0$. However, for $j = 2$ we obtain
	\[
	C_{2,b} = \phi(-b^{1/l})(-b^{1/l}) - \phi(b^{1/l})b^{1/l} = -2b^{1/l}\phi(b^{1/l})\neq 0.
	\]
	Therefore, the asymptotic is the Rosenblatt-type distribution of $-b^{1/l}\phi(b^{1/l})X_2(\Delta).$ 
\end{example}

\begin{example}
	Now, let us study $\int\limits_{\partial\Delta(r)}\chi(S(\eta (x)) >0)\,\mathrm{d}\sigma(x),$ where $S(x) = -x^3 +b^2x$ and $b = \left(2\ln(2)\right)^{1/2}.$ Since 
	\[\int\limits_{\partial\Delta(r)}\chi(S(\eta (x)) >0)\,\mathrm{d}\sigma(x) = \int\limits_{\partial\Delta(r)}\chi(\eta (x) \in (-\infty, -b)\cup (0, b))\,\mathrm{d}\sigma(x),\]
	we can compute coefficients $C_{j, b}$ as follows
	\[C_{1, b} = -\phi(-b)H_{0}(-b) + \phi(0)H_{0}(0) - \phi(b)H_{0}(b) = \phi(0) - 2\phi(b) =0,\]
	\[
	C_{2, b} = -\phi(-b)H_{1}(-b) + \phi(0)H_{1}(0) - \phi(b)H_{1}(b) = b\phi(-b) - b\phi(-b) =0,
	\]
	\[
	C_{3, b} = -\phi(-b)H_{2}(-b) + \phi(0)H_{2}(0) - \phi(b)H_{2}(b) = -\phi(-b)(b^2-1)\] 
	\[- \phi(0) -\phi(-b)(b^2 - 1) =-\phi(0) -\phi(0)(b^2-1) = -b^2\phi(0)\neq 0,
	\]
	because $b = \left(2\ln(2)\right)^{1/2}.$
	
	Thus, in this case the limit distribution has $H\mbox{rank} = 3$.
\end{example}

\begin{example}
	In this example we show how to obtain the Hermite limit distribution with $H\mbox{rank} = 4$.
\begin{lemma}
	For each $p\in(0,1)$ there exist $q>1$, such that $p\phi(p) = q\phi(q).$
\end{lemma}
\begin{proof}
	Note that $(x\phi(x))^{\prime} = \phi(x) - x^2\phi(x) = \phi(x)(1-x^2).$ Thus, $x\phi(x)$ is an increasing function on $(0,1)$ and it is decreasing on $(1,\infty)$. As $x\phi(x) = 0$ for $x = 0$ and $x = +\infty$, then $0<p\phi(p)<\phi(1).$ Because $x\phi(x)$ is a continuous function there is $q > 1$ such that $p\phi(p) = q\phi(q).$
\end{proof}

Note, that $p\phi(p) = q\phi(q),\,p,q>0$ is equivalent to $p^2\phi^2(p) = q^2\phi^2(q)$, i.e. $q$ is a positive solution of the equation
\[-p^2e^{-p^2} = -q^2e^{-q^2}.\]

Thus, $q = \sqrt{{-\rm LambertW_{-1}}\left(-\frac{p^2}{e^{p^2}}\right)},$ where ${\rm lambertW_{-1}}(\cdot)$ is the branch of LambertW function satisfying $\rm LambertW(x) \leq -1,$ $-1/e<x<0,$ see \cite{Cor}.
\end{example}

Let $S(x) = -(x^2 - p^2)(x^2 - q^2).$ Then, $\{x\in\mathbb{R}:S(x)>0\} = (-q,-p)\cup(p,q).$

Let us compute the coefficient $C_{j,0}$.
\[C_{1,0} = \phi(-q) - \phi(-p) + \phi(p) -\phi(q) = 0, \]
\[C_{2,0} = \phi(-q)(-q) - \phi(-p)(-p) + \phi(p)p -\phi(q)q = 2(\phi(p)p -\phi(q)q)=0,\]
\[C_{3,0} = \phi(-q)(q^2 - 1) - \phi(-p)(p^2 - 1) + \phi(p)(p^2 - 1) - \phi(q)(q^2 -1) = 0, \]
\[
C_{4,0} = \phi(-q)(-q^3 + 3q) - \phi(-p)(-p^3 + 3p) +\phi(p)(p^3 -3p) - \phi(q)(q^3 - 3q) 
\]
\[
= \phi(-q)(-q^3) - \phi(-p)(-p^3) + \phi(p)p^3 - \phi(q)q^3 = 2(\phi(p)p^3 - \phi(q)q^3)
\]   
\[
< 2q^2(\phi(p)p - \phi(q)q) = 0.
\]

Therefore $C_{4,0} \neq 0$ and the asymptotic of $\int\limits_{\partial\Delta(r)}\chi(S(\eta (x)) > 0)\,\mathrm{d}\sigma(x)$ when $r\rightarrow\infty$ is the random variable $\frac{C_{4,0}}{4!}X_4(\Delta).$

\section{Rate of convergence}\label{convs}
In this section we investigate rates of convergence of random variables $K_r$ and $K_{r,\kappa}$ to their asymptotic distribution derived in Theorem~\ref{th5}. For readability we will denote Wiener-It\'{o} integrals of rank $\kappa$ by $I_\kappa(f)$, where $f(\cdot)$ is an integrand. For more details about Wiener-It\'{o} integrals and properties of function $f(\cdot)$ one can refer to \cite{Ito, Maj}.
To obtain rates of convergence we will use some fine properties of Hermite-type distributions. The following result was obtained in \cite{New} for $X_\kappa(\Delta)$. Since the proof does not rely on the specific form of $X_\kappa(\Delta)$, this theorem can be easily generalized as follows
\begin{theorem}{\rm \cite{New}}\label{cmb}
	For any $\kappa \in \mathbb{N}$ and an arbitrary positive $\varepsilon$ it holds
	\[\rho\left(I_\kappa(f),I_\kappa(f)+\varepsilon\right)\leq C\varepsilon^{a},\]
	where $a = 1$ if $\kappa < 3$ and $a = 1/\kappa$ if $\kappa \geq 3$.
\end{theorem}

The corollary of Theorem~\ref{th4} is that the limit distribution of the functional $K_r$ does not depend on the ``tail'' $V_r$ in the Hermite expansion of the function $G(r).$ However, in this section we will show that although $V_r$ does not affect the limit distribution it does affect the rate of convergence.

First, let us consider the case where $G(\cdot) = \frac{C_\kappa}{\kappa!}H_\kappa(\cdot)$. Then, $V_r = 0$ and the Hermite rank of $G(\cdot)$ is $\kappa$. We are interested in \[{\rho}\left( \frac{\kappa!\,K_{r,\kappa}}{C_\kappa\,r^{d-1-\frac{\kappa\alpha}{2}}L^{%
		\frac{\kappa}{2}}(r)},X_\kappa(\Delta)\right) = {\rho}\left(X_{\kappa, r}(\Delta),X_\kappa(\Delta)\right).\]
	
By (\ref{spl}) \[
X_{\kappa,r}(\Delta)=
c_2^{\kappa/2}(d,\alpha){\int\limits_{\mathbb{R}^{d\kappa}}}^{\prime}\frac{\mathcal{K}(\lambda _1+\cdots +\lambda _\kappa)Q_r(\lambda _1,\ldots ,\lambda _\kappa)W(\mathrm{d}\lambda
	_1)\ldots W(\mathrm{d}\lambda _\kappa)}{\left\| \lambda _1\right\| ^{(d-\alpha )/2}\cdots
	\left\| \lambda _\kappa\right\| ^{(d-\alpha )/2}},
\]
where $Q(\cdot)$ is defined by (\ref{qqq}). Therefore, ${\rho}\left(X_{\kappa, r}(\Delta),X_\kappa(\Delta)\right)$ is the Kolmogorov's distance between two multiple Wiener-It\'{o} integrals of the rank $\kappa$. To estimate this distance we prove the following result.
\begin{lemma}\label{dst}
	Let $I_\kappa(f_1)$ and $I_\kappa(f_2)$ be two  Wiener-It\`{o} integrals of order $\kappa$, and $f_1,\,f_2$ be symmetric functions in $L_2({\mathbb{R}^d}),\,d\geq1$. Then,
	\begin{equation*}
	\rho\left(I_\kappa(f_1),I_\kappa(f_2) \right) \leq C\|f_1 - f_2\|^{\frac{1}{\kappa+1/2}}, \quad \text{if } \kappa \geq 3,
	\end{equation*}
	and
	\begin{equation*}
	\rho\left(I_\kappa(f_1),I_\kappa(f_2) \right) \leq C\|f_1 - f_2\|^{\frac{2}{3}}, \quad \text{if } \kappa < 3.
	\end{equation*}
\end{lemma}

\begin{proof}
	By applying Lemma~\ref{lem1} to $X=I_\kappa(f_2),$ $Y=I_\kappa(f_1)-I_\kappa(f_2),$ and $Z=I_\kappa(f_2)$ we obtain
	\[
	\rho\left(I_\kappa(f_1),I_\kappa(f_2) \right) \leq {\rho}\left(I_\kappa(f_2)+\varepsilon,I_\kappa(f_2)\right)
	+P\left\{ \left|I_\kappa(f_1)-I_\kappa(f_2)\right| \geq \varepsilon \right\}.
	\]
	Using Theorem~\ref{cmb} we get
	\[
	\rho\left(I_\kappa(f_1),I_\kappa(f_2) \right) \leq C\varepsilon^a
	+P\left\{ \left|I_\kappa(f_1)-I_\kappa(f_2)\right| \geq \varepsilon \right\}
	\]
	\[
	\leq C\varepsilon^a +\varepsilon^{-2} \mathbf{Var}\left(I_\kappa(f_1)-I_\kappa(f_2)\right) \leq C\left(\varepsilon^a + \varepsilon^{-2}\|f_1 - f_2\|^2\right),
	\]
	where $a$ is defined in Theorem~\ref{cmb}.
	By choosing $\varepsilon = \|f_1 - f_2\|^\beta$ we get
	\[
	\rho\left(I_\kappa(f_1),I_\kappa(f_2) \right) \leq C\left(\|f_1 - f_2\|^{\beta a} + \|f_1 - f_2\|^{2-2\beta}\right).
	\]
	Since $\sup\limits_{\beta}\min(a\beta, 2-2\beta) = \frac{2a}{2+a}$, we have	
	\begin{equation*}
	\rho\left(I_\kappa(f_1),I_\kappa(f_2) \right) \leq C\|f_1 - f_2\|^{\frac{2a}{2+a}}.
	\end{equation*}
	Note, that $a = 1$ when $\kappa < 3$, thus 
	\begin{equation*}
	\rho\left(I_\kappa(f_1),I_\kappa(f_2) \right) \leq C\|f_1 - f_2\|^{\frac{2}{3}}, \quad \kappa = 1,2.
	\end{equation*}
	Furthermore, in the case of general $\kappa$, $a = 1/\kappa$ and therefore
	\begin{equation*}
	\rho\left(I_\kappa(f_1),I_\kappa(f_2) \right) \leq C\|f_1 - f_2\|^{\frac{2/\kappa}{2+1/\kappa}} = C\|f_1 - f_2\|^{\frac{1}{\kappa+1/2}}.\qedhere
	\end{equation*}
\end{proof}

\begin{rem}
	For the total variation distance $\rho_{TV}(\cdot)$ it was stated in \cite{DavMart} that \[\rho_{TV}\left(I_\kappa(f_1),I_\kappa(f_2) \right) \leq C\|f_1 - f_2\|^{\frac{1}{\kappa}}.\]
	Since the Kolmogorov's distance can be estimated by the total variation distance (for any random variables $\xi$ and $\eta$ it holds $\rho(\xi, \eta)\leq \rho_{TV}(\xi, \eta)$), result in \cite{DavMart} is an improvement of Lemma~\ref{dst}. But, in \cite{DavMart} only a sketch of a proof is provided,  and \cite{NourPol} questioned the result. Therefore, \cite{NourPol} proved that $\rho_{TV}\left(I_\kappa(f_1),I_\kappa(f_2) \right) \leq C\|f_1 - f_2\|^{\frac{1}{2\kappa}}.$  Note, that this result is worse than ours if we were to use it to estimate Kolmogorov's distance.
	Thus, Lemma~\ref{dst} is presented as a fully proven, self-contained result.
	Unfortunately, Lemma~\ref{lem1} that was used to obtain the result is not applicable for the total variation distance. Hence, our method can not be used for the total variation distance. Therefore, while the result in \cite{NourPol} performs worse in our case, it is more general as a whole.
	
	Recently, for the case of $\kappa = 2$, it was shown in \cite{Zint} that $\rho_{TV}\left(I_2(f_1),I_2(f_2) \right) \leq C\|f_1 - f_2\|.$ This result is an obvious improvement of the existing results. Thus, in the case $\kappa = 2$ we can use it to further sharpen our upper bound. However, we don't see how methods in \cite{Zint} can be used to obtain similar results for an arbitrary $\kappa$ as they heavily rely on the Chi-square expansion of the second order Wiener-It\`{o} integrals, which is not available for $\kappa>2$.
\end{rem}

Now, we apply Lemma~\ref{dst} to obtain the rate of convergence in Theorem~\ref{th5}.

\begin{theorem}
	Let $H 	\mathrm{rank}\,G=\kappa \in \mathbb{N}$ and Assumptions~{\rm{\ref{ass1}}} and {\rm{\ref{ass2}}} hold for $\alpha\in(0, \frac{d-1}{\kappa})$.
	
	If $\tau \in \left(-\frac{d - \kappa\alpha}{2},0\right)$ then for any $%
	\varkappa<\frac{a}{2+a}\min\left(\frac{\alpha(d-1-\kappa\alpha)}{%
		d-1-(\kappa-1)\alpha},\varkappa_1\right)$ 
	\begin{equation*}
		{\rho}\left( \frac{\kappa!\,K_r}{C_\kappa\,r^{d-1-\frac{\kappa\alpha}{2}}L^{%
				\frac{\kappa}{2}}(r)},X_\kappa(\Delta)\right)=o (r^{-\varkappa}),\quad
		r\rightarrow \infty ,
	\end{equation*}
	where 
	$
		\varkappa_1:=\min\left(-2\tau,\frac{1}{\frac{1}{d-2\alpha}+ \dots +\frac{1}{%
				d-\kappa\alpha} +\frac{1}{%
				d-1-\kappa\alpha}}\right)
	$
	and $a$ is the parameter from Theorem~{\rm\ref{cmb}}.
	
	If $\tau=0$ then 
	\begin{equation*}
		{\rho}\left( \frac{\kappa!\,K_r}{C_\kappa\,r^{d-1-\frac{\kappa\alpha}{2}}L^{%
				\frac{\kappa}{2}}(r)},X_\kappa(\Delta)\right)=g^{\frac{2a}{2+a}}(r), \quad
		r\rightarrow \infty.
	\end{equation*}
\end{theorem}

\begin{rem}
	If $\kappa = 1$, then $\varkappa_1=\min\left(-2\tau,d-1-\alpha\right).$
\end{rem}

\begin{rem}
	Note, that for $\tau= 0$ the rate of convergence does not depend on $\alpha$ or $d$. This is due to the reason that parameters $\alpha$ and $d$ affect the power of $r$ in the rate of convergence, but, in the case $\tau = 0$, the function $g(r)$ converges to 0 slower than any power of $r$.
\end{rem}

\begin{proof}
 Since $H {\rm rank}\,G=\kappa,$ it follows that $K_r$ can be represented in the space of squared-integrable random variables $L_2(\Omega)$ as
		\[K_r = K_{r,\kappa}+V_r :=\frac{C_\kappa }{\kappa!}
		\int\limits_{\partial\Delta(r)}H_\kappa (\eta (x))\,\mathrm{d}\sigma(x) + \sum_{j\geq \kappa + 1}\frac{C_j}{j!}
		\int\limits_{\partial\Delta(r)}H_j (\eta (x))\,\mathrm{d}\sigma(x),
		\]
		where $C_j $ are coefficients of the Hermite series of the function $G(\cdot).$
		
		By the proof of Theorem~\ref{th4} (specifically estimates (\ref{I11}) and (\ref{I21})), for sufficiently large $r$
		\[
		\mathbf{Var}\, V_{r} \leq C\,r^{2d-2-\kappa\alpha}L^{\kappa}(r)\left( r^{-\beta_1(d-1-\kappa\alpha-\delta)}
		+
		o\left(r^{-(\alpha-\delta)(1-\beta_1)}\right)\right).
		\]
		Since, by Remark~\ref{eql}, $L_0(\cdot) \sim L(\cdot)$, we can replace $L_0(\cdot)$ by $L(\cdot)$ in the above estimate. Thus, choosing $\beta_1=\frac{\alpha}{d-1-(\kappa - 1)\alpha}$ to minimize the upper bound we get
		\[
		\mathbf{Var}\, V_{r} \leq C r^{2d-2-\kappa\alpha}L^{\kappa}(r)r^{-\frac{\alpha(d-1-\kappa\alpha)}{d-1-(\kappa - 1)\alpha}+\delta}.
		\]
		It follows from Theorem~\ref{cmb} that 
		\[{\rho}\left(X_\kappa(\Delta)+\varepsilon,X_\kappa(\Delta)\right)\le C\varepsilon^a.\]
		
		Applying Chebyshev's inequality and Lemma~\ref{lem1} to $X=X_{\kappa, r}(\Delta)$, $Y=\frac{\kappa!\,V_r}{C_\kappa\,r^{d-1-\frac{\kappa\alpha}{2}}L^{\frac{\kappa}{2}}(r)},$ and $Z=X_{\kappa}(\Delta),$ we get
		\[{\rho}\left( \frac{\kappa!\,K_r}{C_\kappa\,r^{d-1-\frac{\kappa\alpha}{2}}L^{\frac{\kappa}{2}}(r)},X_\kappa(\Delta)\right)={\rho}\left( X_{\kappa, r}(\Delta)+\frac{\kappa!\,V_r}{C_\kappa\,r^{d-1-\frac{\kappa\alpha}{2}}L^{\frac{\kappa}{2}}(r)},X_\kappa(\Delta)\right)\]
		\[
		\le {\rho}\left(X_{\kappa, r}(\Delta),X_\kappa(\Delta)\right)+C\left(\varepsilon^a+ \varepsilon^{-2}\,r^{-\frac{\alpha(d-1-\kappa\alpha)}{d-1-(\kappa - 1)\alpha}+\delta}\right),
		\]
		for a sufficiently large $r.$
		
		Choosing $\varepsilon:=r^{-\frac{\alpha(d-1-\kappa\alpha)}{(2+a)(d-1-(\kappa - 1)\alpha)}}$ to minimize the second term we obtain
		\begin{equation}\label{up11}
		{\rho}\left( \frac{\kappa!\,K_r}{C_\kappa\,r^{d-1-\frac{\kappa\alpha}{2}}L^{\frac{\kappa}{2}}(r)},X_\kappa(\Delta)\right)\le {\rho}\left(X_{\kappa, r}(\Delta),X_\kappa(\Delta)\right)+C\,r^{\frac{-a\alpha(d-1-\kappa\alpha)}{(2+a)(d-1-(\kappa - 1)\alpha)}+\delta}.
		\end{equation}
		\begin{rem}
			As we can see from (\ref{up11}), for a sufficiently large $r$, the upper bound in (\ref{up11}) can be estimated by $C\max\left({\rho}\left(X_{\kappa, r}(\Delta),X_\kappa(\Delta)\right),\,r^{-\frac{a\alpha(d-1-\kappa\alpha)}{(2+a)(d-1-(\kappa - 1)\alpha)}+\delta}\right)$. Here, the part $r^{-\frac{a\alpha(d-1-\kappa\alpha)}{(2+a)(d-1-(\kappa - 1)\alpha)}+\delta}$ appears only when $V_r\neq 0$, i.e. $G(\cdot)\neq \frac{C_\kappa}{\kappa!}H_\kappa(\cdot)$. Depending on the values of parameters $d$, $\kappa$ and $\alpha$ it can considerably affect the rate of convergence. We will discuss it in more details at the end of this section.
		\end{rem}
		Using Lemma~\ref{dst} we get
		\[\rho\left(X_{\kappa, r}(\Delta),X_{\kappa}(\Delta) \right) \leq\]
		\begin{equation}\label{38}
		 C\left[\int\limits_{\mathbb{R}^{\kappa d}}\frac{|\mathcal{K}(\lambda _1+ \dots +\lambda _\kappa)|^2\left(Q_r(\lambda_1,\dots,\lambda_\kappa)-1\right)^2\mathrm{d}\lambda
			_1 \dots \,\mathrm{d}\lambda _\kappa}{\left\| \lambda _1\right\| ^{d-\alpha}\dots\left\| \lambda _\kappa\right\| ^{d-\alpha}}\right]^{\frac{a}{2+a}},
		\end{equation}	
		where
		\[
		Q_r(\lambda _1,\ldots ,\lambda _\kappa): =r^{\kappa(\alpha-d)/2}L^{-\kappa/2}(r)\
		c_2^{-\kappa/2}(d,\alpha)  \left[ \prod\limits_{j=1}^\kappa\left\| \lambda _j\right\| ^{d-\alpha}f\left( \frac{\left\| \lambda _j\right\| }r\right) \right] ^{1/2}.
		\]

		Let us rewrite the integral in (\ref{38}) as the sum of two integrals $I_3$ and $I_4$ with the integration regions $A(r):=\{(\lambda _1,\dots,\lambda _\kappa)\in\mathbb{R}^{\kappa d}:\ \max\limits_{i = \overline{1,\kappa}}(||\lambda _i||)\le r^\gamma\}$ and $\mathbb{R}^{\kappa d}\setminus A(r)$ respectively, where $\gamma\in(0,1).$ Our intention is to use the monotone equivalence property of regularly varying functions in the regions~$A(r).$
		
		First we consider the case of $(\lambda _1,\dots\lambda _\kappa)\in A(r).$  By Assumption~\ref{ass2} and the inequality \[\left|\sqrt{\prod\limits_{i = 1}^{\kappa}x_i}-1\right|\le \sum\limits_{i = 1}^{\kappa}\left|x^{\frac{\kappa}{2}}_i - 1\right|\] we obtain
		
		\[
		|Q_r(\lambda _1,\dots, \lambda _2)-1| = \left|\sqrt{\prod\limits_{j = 1}^{\kappa}\frac{L\left( \frac{r}{\left\| \lambda_j\right\| }\right)}{L(r)}}-1\right| \le \sum\limits_{j = 1}^{\kappa}\left|\frac{L^{\frac{\kappa}{2}}\left( \frac{r}{\left\| \lambda_j\right\| }\right)}{L^{\frac{\kappa}{2}}(r)}-1\right|.\]

		By Remark~\ref{rem0}, if $||\lambda _j||\in (1, r^\gamma),$ $j=\overline{1,\kappa},$ then for arbitrary $\delta_1 > 0$ and sufficiently large $r$ we get
		\[\left|1-\frac{L^{\frac{\kappa}{2}}\left( \frac{r}{\left\| \lambda _j\right\| }\right)}{L^{\frac{\kappa}{2}}(r)} \right|=\frac{L^{\frac{\kappa}{2}}\left( \frac{r}{\left\| \lambda _j\right\| }\right)}{L^{\frac{\kappa}{2}}(r)} \cdot\left|1-\frac{L^{\frac{\kappa}{2}}(r)}{L^{\frac{\kappa}{2}}\left( \frac{r}{\left\| \lambda _j\right\| }\right)} \right|\le C\,\frac{L^{\frac{\kappa}{2}}\left( \frac{r}{\left\| \lambda _j\right\| }\right)}{L^{\frac{\kappa}{2}}(r)}g\left( \frac{r}{\left\| \lambda _j\right\|}\right)\]
		\[\times\left\| \lambda _j\right\|^{\delta_1} h_{\tau}(\left\| \lambda _j\right\|)= C\,  \left\| \lambda _j\right\|^{\delta_1} h_{\tau}(\left\| \lambda _j\right\|)g(r)\frac{g\left( \frac{r}{\left\| \lambda _j\right\|}\right)}{g(r)}\left(\frac{L\left( \frac{r}{\left\| \lambda _j\right\| }\right)}{L(r)}\right)^{\frac{\kappa}{2}}.\]
		
		For any positive $\beta_2$ and $\beta_3$, applying Theorem~1.5.6 \cite{bin} to $g(\cdot)$ and $L(\cdot)$ and using the fact that $h_{\tau}\left(\frac{1}{t}\right) = -\frac{1}{t^{\tau}}h(t)$ we obtain
		\begin{equation}\label{gr1}
		\left|1-\frac{L^{\frac{\kappa}{2}}\left( \frac{r}{\left\| \lambda _j\right\| }\right)}{L^{\frac{\kappa}{2}}(r)} \right|\le C\,  \left\| \lambda _j\right\|^{\delta_1 + \frac{\kappa\beta_2}{2} + \beta_3 } \frac{h_{\tau}(\left\| \lambda _j\right\|)}{\left\| \lambda _j\right\|^{ \tau}}g(r) = C\,   \left\| \lambda _j\right\|^{\delta} h_{\tau}\left(\frac{1}{\left\| \lambda _j\right\|}\right)g(r). 
		\end{equation}

		By  Remark~\ref{rem0} for $||\lambda _j||\le 1,$ $j=\overline{1,\kappa}$, and arbitrary $\delta > 0,$ we obtain
		\begin{equation}\label{gr2}
		\left|1-\frac{L^{\frac{\kappa}{2}}\left( \frac{r}{\left\| \lambda _j\right\| }\right)}{L^{\frac{\kappa}{2}}(r)} \right| \le C\,   \left\| \lambda _j\right\|^{-\delta} h_{\tau}\left(\frac{1}{\left\| \lambda _j\right\|}\right)g(r).
		\end{equation}
		
		Hence, by  (\ref{gr1}) and (\ref{gr2})
		\[
		|Q_r(\lambda _1,\dots \lambda _\kappa)-1|^2 \le k \sum\limits_{j = 1}^{\kappa}\left|\frac{L^{\frac{\kappa}{2}}\left( \frac{r}{\left\| \lambda_j\right\| }\right)}{L^{\frac{\kappa}{2}}(r)}-1\right|^{2}\] 
		\[\le C\sum\limits_{j = 1}^{\kappa} h^{2}_{\tau}\left(\frac{1}{\left\| \lambda _j\right\|}\right)g^{2}(r) \max\left(\left\| \lambda _j\right\|^{-\delta}, \left\| \lambda _j\right\|^{\delta}\right),
		\]
		for $(\lambda _1,\dots \lambda _\kappa)\in A(r)$.

		Notice, that in the case $\tau = 0$ for any  $\delta>0$ there exists $C>0$ such that $h_0(x)=\ln(x)<Cx^{\delta},\,x \ge 1$, and $h_0(x)=\ln(x)<Cx^{-\delta},\,x < 1$. Hence, by Lemma~\ref{finint} for $-\tau \le \frac{d -\kappa\alpha}{2}$ we get
		\[\int\limits_{A(r)\cap [0,1]^{\kappa d}}\frac{ h^{2}_{\tau}\left(\frac{1}{\left\| \lambda _j\right\|}\right)\max\left(\left\| \lambda_j\right\|^{-\delta},\left\|\lambda_j\right\|^{\delta}\right)\left|\mathcal{K}\left(\sum\limits_{i=1}^{\kappa}\lambda_i\right)\right|^2\,\mathrm{d}\lambda_1 \dots \mathrm{d}\lambda _\kappa }{\left\| \lambda _1\right\| ^{d-\alpha}\dots\left\| \lambda _\kappa\right\| ^{d-\alpha}} < \infty.\]

		Therefore, we obtain for sufficiently large $r$
		\[I_3\le C\,g^{2}(r) \sum\limits_{j = 1}^{\kappa}\int\limits_{A(r)\cap \mathbb{R}^{\kappa d}}\frac{ h^{2}_{\tau}\left(\frac{1}{\left\| \lambda _j\right\|}\right)\cdot \max\left(\left\| \lambda _j\right\|^{-\delta}, \left\| \lambda _j\right\|^{\delta}\right) }{\left\| \lambda _1\right\| ^{d-\alpha}\dots\left\| \lambda _\kappa\right\| ^{d-\alpha}}\]
		\[
		\times|\mathcal{K}(\lambda _1+\dots \lambda _\kappa)|^2\,\mathrm{d}\lambda
		_1 \dots \mathrm{d}\lambda _\kappa\le C\,g^{2}(r) \int\limits_{A(r)\cap \mathbb{R}^{\kappa d}}\frac{ h^{2}_{\tau}\left(\frac{1}{\left\| \lambda _1\right\|}\right) }{\left\| \lambda _1\right\| ^{d-\alpha}\dots\left\| \lambda _\kappa\right\| ^{d-\alpha}}\] 
		\begin{equation}\label{del}
		\times\max\left(\left\| \lambda _1\right\|^{-\delta}, \left\| \lambda _1\right\|^{\delta}\right)|\mathcal{K}(\lambda _1+\dots \lambda _\kappa)|^2\,\mathrm{d}\lambda
		_1 \dots \mathrm{d}\lambda _\kappa \le C\,g^{2}(r).
		\end{equation}

		It follows from Assumption~\ref{ass2} and the specification  of the estimate (\ref{upper})  in the proof of Theorem~\ref{th5}  that for each positive $\delta$ there exists $r_0>0$ such that for all $r\ge r_0,$   $(\lambda _1,\dots,\lambda _\kappa)\in B_{(1,\mu_2,\dots,\mu_\kappa)}=\{(\lambda _1,\dots,\lambda _\kappa)\in\mathbb{R}^{\kappa d}: ||\lambda_j||\le 1,\ \mbox{if}\ \mu_j=-1,\ \mbox{and } ||\lambda_j||> 1, \ \mbox{if}\ \mu_j=1, j=\overline{1,k}\},$ and $\mu_j\in \{-1,1\},$ it holds
		\[
		\frac{|\mathcal{K}(\lambda _1+\dots +\lambda _\kappa)|^2\left(Q_r(\lambda _1,\dots\lambda _\kappa)-1\right)^2}{\left\| \lambda _1\right\| ^{d-\alpha}\dots
			\left\| \lambda _\kappa\right\| ^{d-\alpha}}\le \frac{C\, |\mathcal{K}(\lambda _1 +\dots+\lambda _\kappa)|^2}{\left\| \lambda _1\right\| ^{d-\alpha}
			\dots\left\| \lambda _\kappa\right\| ^{d-\alpha}}\]
		\[+C\,\frac{|\mathcal{K}(\lambda _1+\dots+\lambda _\kappa)|^2}{\left\| \lambda _1\right\| ^{d-\alpha-\delta}
			\left\| \lambda _2\right\| ^{d-\alpha-\mu_2\delta}\dots\left\| \lambda _\kappa\right\| ^{d-\alpha-\mu_\kappa\delta}}.
		\]
		
		Since the integrands are non-negative, we can estimate $I_4$ as it is shown below
		\[ I_4\le \kappa\int\limits_{\mathbb{R}^{(\kappa-1)d}}\int\limits_{||\lambda _1||> r^\gamma}\frac{|\mathcal{K}(\lambda _1+\dots+\lambda _\kappa)|^2\left(Q_r(\lambda_1,\dots,\lambda_\kappa)-1\right)^2\mathrm{d}\lambda
			_1 \dots\mathrm{d}\lambda _\kappa}{\left\| \lambda _1\right\| ^{d-\alpha}\dots\left\| \lambda _\kappa\right\| ^{d-\alpha}}\]
		\[ \le C\int\limits_{\mathbb{R}^{(\kappa-1)d}}\int\limits_{||\lambda _1||> r^\gamma}\frac{|\mathcal{K}(\lambda _1+\dots+\lambda _2)|^2\,\mathrm{d}\lambda
			_1 \dots\mathrm{d}\lambda _\kappa}{\left\| \lambda _1\right\| ^{d-\alpha}\dots\left\| \lambda _\kappa\right\| ^{d-\alpha}}\]
		\[+\,  C\sum\limits_{\lfrac{\mu_i\in\{0,1,-1\}}{i \in \overline{1,\kappa}}} \int\limits_{\mathbb{R}^{(\kappa-1)d}}\int\limits_{||\lambda _1||> r^\gamma}\frac{|\mathcal{K}(\lambda _1+\dots +\lambda _\kappa)|^2\mathrm{d}\lambda_1 \dots\mathrm{d}\lambda_\kappa}{\left\| \lambda _1\right\| ^{d-\alpha-\delta}\left\| \lambda _2\right\| ^{d-\alpha-\mu_2\delta}\dots\left\| \lambda _\kappa\right\| ^{d-\alpha-\mu_\kappa\delta}}\]
		\begin{equation}\label{midest}
		\le C\max_{\lfrac{ \mu_i\in\{0,1,-1\}}{i \in \overline{2,\kappa} }}\int\limits_{\mathbb{R}^{(\kappa-1)d}}\int\limits_{||\lambda _1||> r^\gamma}\frac{|\mathcal{K}(\lambda_1 +\dots +\lambda _\kappa)|^2\mathrm{d}\lambda_1 \dots\mathrm{d}\lambda_\kappa}{\left\| \lambda _1\right\| ^{d-\alpha-\delta}\left\|\lambda _2\right\| ^{d-\alpha-\mu_2\delta}\dots\left\| \lambda _\kappa\right\| ^{d-\alpha-\mu_\kappa\delta}}.
		\end{equation}
		Replacing $\lambda_1 + \lambda_2$ by $u$ we obtain
		\[I_4\le C\max_{\lfrac{ \mu_i\in\{0,1,-1\}}{i \in \overline{2,\kappa} }}\int\limits_{\mathbb{R}^{(\kappa-1)d}}\int\limits_{||\lambda _1||> r^\gamma}\frac{|\mathcal{K}(u + \lambda _3+\dots +\lambda _\kappa)|^2}{\left\| \lambda _1\right\| ^{d-\alpha-\delta}\left\|u - \lambda _1\right\| ^{d-\alpha-\mu_2\delta}}\]
		\[
		\times\frac{\mathrm{d}\lambda_1\mathrm{d}u\mathrm{d}\lambda_3 \dots\mathrm{d}\lambda_\kappa}{\left\| \lambda _3\right\| ^{d-\alpha-\mu_3\delta}\dots\left\| \lambda _\kappa\right\| ^{d-\alpha-\mu_\kappa\delta}}\le C\max\limits_{\lfrac{ \mu_i\in\{0,1,-1\}}{i \in \overline{2,\kappa} }} \int\limits_{\mathbb{R}^{(\kappa-1)d}}\frac{1}{\left\|u\right\| ^{d-2\alpha-(\mu_2+1)\delta}}\]
		\[\times\frac{|\mathcal{K}(u + \lambda _3+\dots +\lambda _\kappa)|^2}{\left\| \lambda _3\right\| ^{d-\alpha-\mu_3\delta}\dots\left\| \lambda _\kappa\right\| ^{d-\alpha-\mu_\kappa\delta}}\int\limits_{\|\lambda _1\|> \frac{r^\gamma}{\|u\|}}\frac{\mathrm{d}\lambda_1 \mathrm{d}u\,\mathrm{d}\lambda_3 \dots\mathrm{d}\lambda_\kappa}{\left\| \lambda _1\right\| ^{d-\alpha-\delta}\left\|\frac{u}{\left\|u\right\|} - \lambda _1\right\| ^{d-\alpha-\mu_2\delta}}.\]
		Taking into account that for $\delta\in (0,\min(\alpha,{d}/{\kappa}-\alpha))$
		\[\sup_{u\in\mathbb{R}^{d}\setminus\{0\}}\int\limits_{\mathbb{R}^{d}}\frac{\,\mathrm{d}\lambda
			_1}{\left\| \lambda _1\right\| ^{d-\alpha-\delta}\left\|\frac{u}{\|u\|}- \lambda _1\right\| ^{d-\alpha-\mu_2\delta}}\le C,\]
		we obtain
		\[I_4\le C\max_{\lfrac{ \mu_i\in\{0,1,-1\}}{i \in \overline{3,\kappa} }}\int\limits_{\mathbb{R}^{(\kappa-2)d}}\left[\max_{\mu_2\in\{0,1,-1\}}\int\limits_{||u||\le r^\gamma_0}\frac{|\mathcal{K}(u + \lambda _3+\dots +\lambda _\kappa)|^2}{\left\|u\right\| ^{d-2\alpha-(\mu_2+1)\delta}}\right.\]
		\[\times\frac{\mathrm{d}\lambda_3 \dots\mathrm{d}\lambda_\kappa}{\left\| \lambda _3\right\| ^{d-\alpha-\mu_3\delta}\dots\left\| \lambda _\kappa\right\| ^{d-\alpha-\mu_\kappa\delta}}\left.\int\limits_{||\lambda _1||> r^{\gamma-\gamma_0}}\frac{\mathrm{d}\lambda_1 \mathrm{d}u}{\left\| \lambda _1\right\| ^{d-\alpha-\delta}\left\|\frac{u}{\left\|u\right\|} - \lambda _1\right\| ^{d-\alpha-\mu_2\delta}}\right.\]
		\[+\left.\max_{\mu_i\in\{0,1,-1\}}\int\limits_{||u||> r^\gamma_0}\frac{|\mathcal{K}(u + \lambda _3+\dots +\lambda _\kappa)|^2\mathrm{d}u\,\mathrm{d}\lambda_3 \dots\mathrm{d}\lambda_\kappa}{\left\|u\right\| ^{d-2\alpha-(\mu_2+1)\delta}\left\| \lambda _3\right\| ^{d-\alpha-\mu_3\delta}\dots\left\| \lambda _\kappa\right\| ^{d-\alpha-\mu_\kappa\delta}}\right],\]
		where $\gamma_0\in (0,\gamma).$
		
		By Lemma~\ref{finint}, there exists $r_0>0$ such that for all $r\ge r_0$ the first summand is bounded by
		\[C\max_{\mu_2\in\{0,1,-1\}}\int\limits_{||u||\le r^\gamma_0}\frac{|\mathcal{K}(u + \lambda _3+\dots +\lambda _\kappa)|^2\mathrm{d}u\mathrm{d}\lambda_3 \dots\mathrm{d}\lambda_\kappa}{\left\|u\right\| ^{d-2\alpha-(\mu_2+1)\delta}\left\| \lambda _3\right\| ^{d-\alpha-\mu_3\delta}\dots\left\| \lambda _\kappa\right\| ^{d-\alpha-\mu_\kappa\delta}}\]
		\[\times\int\limits_{||\lambda _1||> r^{\gamma-\gamma_0}}\frac{\mathrm{d}\lambda_1} {\left\| \lambda _1\right\| ^{2d-2\alpha-\delta-\mu_2\delta}}\le Cr^{-(\gamma-\gamma_0)(d-2\alpha-2\delta)}.\]
		
		Therefore, for sufficiently large $r,$
		\[I_4\le Cr^{-(\gamma-\gamma_0)(d-2\alpha-2\delta)}\]
		\[+C\max_{\lfrac{ \mu_i\in\{0,1,-1\}}{i \in \overline{3,\kappa} }} \int\limits_{\mathbb{R}^{(\kappa-2)d}}\int\limits_{||u||> r^{\gamma_0}}\frac{|\mathcal{K}(u + \lambda _3+\dots +\lambda _\kappa)|^2\mathrm{d}u\mathrm{d}\lambda_3 \dots\mathrm{d}\lambda_\kappa}{\left\|u\right\| ^{d-2\alpha-2\delta}\left\| \lambda _3\right\| ^{d-\alpha-\mu_3\delta}\dots\left\| \lambda _\kappa\right\| ^{d-\alpha-\mu_\kappa\delta}}.
		\]
		
		Notice that the second summand here coincides with \rm(\ref{midest}) if $\kappa$ is replaced by $\kappa-1$. Thus, we can repeat the above procedure $\kappa-2$ more times and get the result
		\begin{equation}\label{+}
		I_4\le Cr^{-(\gamma-\gamma_0)(d-2\alpha-2\delta)}+ \dots + Cr^{-(\gamma_{\kappa-3}-\gamma_{\kappa-2})(d-\kappa\alpha-\kappa\delta)} + C\hspace{-1em}\int\limits_{\|u\|> r^{\gamma_{\kappa-2}}}
		\frac{|\mathcal{K}(u)|^2\,\mathrm{d}u}{\left\| u\right\| ^{d-\kappa\alpha-\kappa\delta}},
		\end{equation}
		where $\gamma>\gamma_0>\gamma_1>\dots>\gamma_{\kappa-2}.$
		
		Using integration formula for polar coordinates and estimate (\ref{l2a}) we obtain
		\[
		\int\limits_{\|u\|> r^{\gamma_{\kappa-2}}}
		\frac{|\mathcal{K}(u)|^2\,\mathrm{d}u}{\left\| u\right\| ^{d-\kappa\alpha-\kappa\delta}}\leq  \int\limits_{r^{\gamma_{\kappa-2}}}^{\infty}t^{d-1}\int\limits_{S_{d-1}(1)}\frac{|\mathcal{K}(\omega t)|^2}{t^{d - \kappa\alpha-\kappa\delta}}\,\mathrm{d}\omega\mathrm{d}t
		\leq
		C\int\limits_{r^{\gamma_{\kappa-2}}}^{\infty}
		\frac{\mathrm{d}t}{t^{d-\kappa(\alpha+\delta)}}
		\]
		\begin{equation}\label{uppp}	    
		\leq C\,r^{-\gamma_{\kappa-2} (d - 1-\kappa(\alpha+\delta))}.
		\end{equation}
		
		Now let us consider the case $\tau<0$. In this case by Theorem~1.5.6 \cite{bin} for any $\delta>0$ we can estimate $g(r)$ as follows 
		\begin{equation}\label{less0}
		g(r)\le C\,r^{\tau+\delta}.
		\end{equation}
		Combining estimates (\ref{up11}), (\ref{38}), (\ref{del}), (\ref{+}), (\ref{uppp}),(\ref{less0})  we obtain
		\[{\rho}\left( \frac{\kappa!\,K_r}{C_\kappa\,r^{d-\frac{\kappa\alpha}{2}}L^{\frac{\kappa}{2}}(r)},X_\kappa(\Delta)\right)\le C\left(r^{-\frac{a\alpha(d-1-\kappa\alpha)}{(2+a)(d-1-(\kappa-1)\alpha)}+\delta}+\left(
		r^{2\tau + 2\delta}+ r^{-(\gamma-\gamma_0)(d-2\alpha-2\delta)}\right.\right.\]
		\[ \left.+ \dots + r^{-(\gamma_{\kappa-3}-\gamma_{\kappa-2})(d-\kappa\alpha-\kappa\delta)}\left.+ r^{-\gamma_{\kappa-2} (d - 1 -\kappa\alpha-\kappa\delta)}\right)^{\frac{a}{2+a}}\right).
		\]

		Therefore, for any  $\tilde\varkappa_1\in (0,\varkappa_0)$ one can choose a sufficiently small $\delta>0$ such that 
		\begin{equation}\label{bou}
		{\rho}\left( \frac{\kappa!\,K_r}{C_\kappa\,r^{d-1-\frac{\kappa\alpha}{2}}L^{\frac{\kappa}{2}}(r)},X_\kappa(\Delta)\right)\le Cr^\delta\left(r^{-\frac{a\alpha(d-1-\kappa\alpha)}{(2+a)(d-1-(\kappa-1)\alpha)}}+ r^{-\frac{a\tilde\varkappa_1}{2+a}}\right),
		\end{equation}
		where  \[\varkappa_0:=\sup_{1>\gamma>\gamma_0>\dots>\gamma_{\kappa-1}=0}\min\left(-2\tau,(\gamma-\gamma_0)(d-2\alpha), \dots , \right.\]
		\[\left.(\gamma_{\kappa-3}-\gamma_{\kappa-2})(d-\kappa\alpha), \left(\gamma_{\kappa-2}-\gamma_{\kappa-1}\right) (d-1-\kappa\alpha)\right).
		\]
		\begin{lemma}\label{max}
			Let $\mathbf{x} =$ $(x_0,\dots,x_n)\in \mathbb{R}^{n+1}_{+}$ be some fixed vector and $\mathbf{\Gamma} = \left\{\gamma = (\gamma_1,\dots,\gamma_{n+1}) \left|\right.\right.$
			$\left.b=\gamma_0>\gamma_1>\dots>\gamma_{n+1}=0\right\}$.
			
			The function $G(\gamma) = \min\limits_{i}\left(\gamma_i-\gamma_{i+1}\right)x_i$ 
			reaches its maximum at $\bar{\gamma}=(\bar{\gamma}_0,\dots,\bar{\gamma}_{n+1})\in \mathbf{\Gamma}$ such that for any $0 \le i \le n$ it holds
			\begin{equation}\label{midpr}
			\left(\bar{\gamma_i}-\bar{\gamma}_{i+1}\right)x_i = \left(\bar{\gamma}_{i+1} -\bar{\gamma}_{i+2}\right)x_{i+1}.
			\end{equation}	
		\end{lemma}
		
		\begin{proof}
			Let us show that any deviation of $\gamma$ from $\bar{\gamma}$ leads to a smaller result.
			Consider a vector $\hat{\gamma}$ such that for some $i \in \overline{1,n}$ and some $\varepsilon>0$ the following relation is true 
			\[\hat{\gamma_i}-\hat{\gamma}_{i+1} = \bar{\gamma_i}-\bar{\gamma}_{i+1} + \varepsilon.\] 
			Since $\sum\limits_{i=0}^{n}\hat{\gamma_i}-\hat{\gamma}_{i+1}= \hat{\gamma_0}-\hat{\gamma}_{n+1} = b$ we can conclude that there exist some $j \neq i,\, j \in \overline{1,n},$ and $\varepsilon_1>0$ such that $\hat{\gamma_j}-\hat{\gamma}_{j+1} = \bar{\gamma_j}-\bar{\gamma}_{j+1} - \varepsilon_1$. 
			
			Obviously, in this case 
			\[G(\hat{\gamma}) \le \left(\hat{\gamma_j}-\hat{\gamma}_{j+1}\right)x_j = \left(\bar{\gamma_j}-\bar{\gamma}_{j+1} - \varepsilon_1\right)x_j =\left(\bar{\gamma_j}-\bar{\gamma}_{j+1}\right)x_j - \varepsilon_1 x_j\]
			Since $\varepsilon_1 >0$ and $x_j >0$ it follows from \rm(\ref{midpr}) that
			\[G(\hat{\gamma}) \le \left(\bar{\gamma_j}-\bar{\gamma}_{j+1}\right)x_j - \varepsilon_1 x_j < \left(\bar{\gamma_j}-\bar{\gamma}_{j+1}\right)x_j = G(\bar{\gamma}).\]
			So it's clearly seen that any deviation from $\bar{\gamma}$ will yield a smaller result.
		\end{proof}
		
		Note, that for fixed $\gamma\in(0,1)$ by Lemma~\ref{max}
		\[\sup_{\gamma>\gamma_0>\dots>\gamma_{\kappa-1}=0}\min\left((\gamma-\gamma_0)(d-2\alpha), \dots ,(\gamma_{\kappa-3}-\gamma_{\kappa-2})(d-\kappa\alpha),\left(\gamma_{\kappa-2}-\gamma_{\kappa-1}\right) (d-\kappa\alpha)\right)\]
		\[=\frac{\gamma}{\frac{1}{d-2\alpha}+ \dots + \frac{1}{d-\kappa\alpha}+\frac{1}{d-1-\kappa\alpha}}
		\]
		and 
		\[\sup_{\gamma\in(0,1)}\frac{\gamma}{\frac{1}{d-2\alpha}+ \dots +\frac{1}{d-\kappa\alpha}+ \frac{1}{d-1-\kappa\alpha}}=\frac{1}{\frac{1}{d-2\alpha}+ \dots +\frac{1}{d-\kappa\alpha}+ \frac{1}{d-1-\kappa\alpha}}.\]
	Thus, $\varkappa_0=\varkappa_1$, and from (\ref{bou}) the first statement of the theorem follows. 
		
		Now let us consider the case $\tau = 0$.
		In this case by Theorem~1.5.6 \cite{bin} for any $s>0$ and sufficiently large $r$
		\begin{equation}\label{eq0}
		g(r)>r^{-s}.
		\end{equation}
		By combining estimates (\ref{up11}), (\ref{38}), (\ref{del}), (\ref{+}), (\ref{uppp}) and using (\ref{eq0}) to replace all powers of $r$ by $g^2(r)$  we obtain		
		\[{\rho}\left( \frac{\kappa!\,K_r}{C_\kappa\,r^{d-1-\frac{\kappa\alpha}{2}}L^{\frac{\kappa}{2}}(r)},X_\kappa(\Delta)\right)\le C\left(g^2(r)+ g^{\frac{2a}{2+a}}(r)\right).\]
		
		Since $a\leq 1,$
		it follows that
		\[{\rho}\left( \frac{\kappa!\,K_r}{C_\kappa\,r^{d-1-\frac{\kappa\alpha}{2}}L^{\frac{\kappa}{2}}(r)},X_\kappa(\Delta)\right)\le Cg^{\frac{2a}{2+a}}(r).\]
		This proves the second statement of the theorem.
        \end{proof}
\begin{rem}
	For example, for $g(x) = \frac{1}{\ln(x)}$ in Remark~\ref{log} we obtain
	\[
	{\rho}\left( \frac{\kappa!\,K_r}{C_\kappa\,r^{d-1-\frac{\kappa\alpha}{2}}L^{\frac{\kappa}{2}}(r)},X_\kappa(\Delta)\right) \leq C\ln^{-\frac{2a}{2+a}}(r).
	\]
\end{rem}

	Let us study how the upper bounds in the rate of convergence perform depending on their parameters. 
	
	When $\tau = 0$ it is quite straightforward to see that for $g(r)$ close to 0 the upper bound decreases as $a$ increases.
	
	For the case $\tau < 0$, let us investigate the upper bound of $\varkappa$ as a function of $\alpha$. 
	\[
	\varkappa <\frac{a}{2+a}\min\left(\frac{\alpha(d-1-\kappa\alpha)}{%
		d-1-(\kappa-1)\alpha},\varkappa_1\right) = \frac{a}{2+a}\min\left(\frac{1}{%
		\frac{1}{\alpha}+\frac{1}{d-1-\kappa\alpha}},\varkappa_1\right).
	\]
	
	Since $\varkappa_1 > 0$, it is obvious that if $\alpha\rightarrow 0$ or $\alpha\rightarrow \frac{d-1}{\kappa}$ the upper bound decreases to~0. Thus, as expected, for these values of $\alpha$ our estimate does not provide a good rate of convergence. 
	
	Let us determine $\alpha$ that corresponds to the best possible bound. We have to compare $\frac{1}{\frac{1}{\alpha}+\frac{1}{d-1-\kappa\alpha}}$ and $\frac{1}{\frac{1}{d-2\alpha}+ \dots +\frac{1}{d-\kappa\alpha} +\frac{1}{d-1-\kappa\alpha}}.$
	Notice, that $\frac{1}{\alpha}$ is a decreasing function of $\alpha$, but  $\frac{1}{d-2\alpha}+ \dots+\frac{1}{d-\kappa\alpha}$ is an increasing function of $\alpha$ on $(0,\frac{d-1}{\kappa}).$ Also, for $\alpha\rightarrow\frac{d-1}{\kappa}$ we get $\frac{1}{\alpha}\rightarrow\frac{\kappa}{d-1},$ and 
	\[\frac{1}{d-2\alpha}+ \dots+\frac{1}{d-\kappa\alpha} \rightarrow \frac{1}{(d-1)(1-\frac{2}{\kappa}) +1}
	+ \dots+\frac{1}{(d-1)(1-\frac{\kappa}{\kappa}) +1}.\]
	
	Hence, we have two cases depending on the values of parameters $\kappa$ and $d$.
	
	{\it Case 1}. If \begin{equation}\label{bb}\frac{\kappa}{d-1} < \frac{1}{(d-1)(1-\frac{2}{\kappa}) +1}
	+ \dots+\frac{1}{(d-1)(1-\frac{\kappa}{\kappa}) +1}\end{equation}
	then there exists $\alpha^{\star} = \arg(\frac{1}{\alpha} = \frac{1}{d-2\alpha}+ \dots+\frac{1}{d-\kappa\alpha})$ that provides the best possible bound
	\[
	\varkappa <	\frac{a}{2+a}\min\left(\frac{1}{%
		\frac{1}{\alpha^{\star}}+\frac{1}{d-1-\kappa\alpha^{\star}}},-2\tau\right).
	\]
	
	{\it Case 2}. If condition (\ref{bb}) doesn't hold, then, regardless of $\alpha$, we have $\frac{1}{%
	\frac{1}{\alpha}+\frac{1}{d-1-\kappa\alpha}} < \varkappa_1$ and the upper bound is $\frac{a}{2+a}\left(\frac{1}{
		\frac{1}{\alpha}+\frac{1}{d-1-\kappa\alpha}}\right)$.
	Choosing $\alpha = \frac{d-1}{\kappa + 1}$ that maximizes this expression we get the best bound
	\[
	\varkappa < \frac{a}{2+a}\min\left(\frac{d-1}{2 + 2\kappa}, -2\tau\right).
    \]
	 Since the bound does not depend on the right part of $\varkappa_1$ in this case, then the rate of converge is determined only by the tail of the Hermite expansion of the function $G(\cdot)$ and by a parameter of the random field $\tau$ introduced in Assumption~\ref{ass2}. 
	\begin{rem}
	If $\kappa = 1$ then there are no such $d$ that condition (\ref{bb}) holds true and only Case~2 is applicable.
	\end{rem}
	 
	\begin{example}
		If $d=2$ then for any $\kappa\in \mathbb{N}$ it holds $\frac{1}{2-\frac{2}{\kappa}}
		+ \dots+\frac{1}{2-\frac{\kappa -1}{\kappa}} + 1< \kappa$. Therefore, only Case~2 is possible and for any $\kappa$ the best bound is
	\[
	\varkappa < \frac{a}{2+a}\min\left(\frac{1}{2 + 2\kappa}, -2\tau\right).
	\]
	\end{example}
	\begin{example}
		If $\kappa =2$ and $d = 4$ then condition (\ref{bb}) holds and $\alpha^{\star} = \arg(\frac{1}{\alpha} = \frac{1}{4-2\alpha}) = 4/3$. Using the fact that $a = 1$ for $\kappa =2$, we get that the best bound is $\varkappa < \min\left(4/45, -2\tau/3\right)$.
	\end{example}


\begin{thebibliography}{100}
	
	\bibitem{Douk} Doukhan, P., Oppenheim, G., Taqqu, M.S. (2003). \textit{Long-Range Dependence: Theory and Applications}. Boston, Birkh\"{a}user.
	
	\bibitem{Iv} Ivanov, A.V., Leonenko, N.N. (1989).  \textit{Statistical Analysis of Random Fields}. Dordrecht, Kluwer Academic Publishers.
	
	\bibitem{Wack} Wackernagel, H. (1998). \textit{Multivariate Geostatistics}. Berlin, Springer-Verlag.
	
	\bibitem{Anh} Anh, V.V., Leonenko, N.N., Ruiz-Medina, M.D. (2013). Macroscaling limit theorems for filtered spatiotemporal random fields. \textit{Stochastic Anal. Appl.} 31(3):460--508. DOI:10.1080/07362994.2013.777280
	
	\bibitem{Ole} Olenko, A. (2013). Limit theorems for weighted functionals of cyclical long-range dependent random fields. \textit{Stochastic Anal. Appl.} 31(2):199-213.
	DOI:10.1080/07362994.2013.741410
	
	\bibitem{LeoRMT} Leonenko, N.N., Ruiz-Medina, M.D., Taqqu, M.S. (2017). Rosenblatt distribution subordinated to Gaussian random fields with long-range dependence. \textit{Stochastic Anal. Appl.} 35(1):144-177.
	DOI:10.1080/07362994.2016.1230723
	
	\bibitem{Dob} Dobrushin, R.L., Major,  P. (1979) Non-central limit theorems for nonlinear functionals of Gaussian fields. \textit{Z. Wahrsch. Verw. Gebiete.} 50(1):27--52.
	DOI:10.1007/BF00535673
	
	\bibitem{Taq1} Taqqu, M.S. (1975) Weak convergence to fractional Brownian motion and to the Rosenblatt process. \textit{Z. Wahrsch. Verw. Gebiete} 31(4):287--302.
	DOI:10.1007/BF00532868
	
	\bibitem{Taq2} Taqqu, M.S. (1979) Convergence of integrated processes of arbitrary {H}ermite rank, \textit{Z. Wahrsch. Verw. Gebiete} 50(1):53--83.
	DOI:10.1007/BF00535674
	
	\bibitem{New} Anh, V., Leonenko, N., Olenko, A., Vaskovych, V.  {On rate of convergence in non-central limit theorems}, 	arXiv:1703.05900
	
	\bibitem{Souj} Leonenko, N., Olenko, A. (2014). Sojourn measures of Student and Fisher-Snedecor random fields. \textit{Bernoulli} 20(3):1454--1483.
	DOI:10.3150/13-BEJ529
	
	\bibitem{Bul} Bulinski, A., Spodarev, E.,  Timmermann, F. (2012). Central limit theorems for the excursion set volumes of weakly dependent random fields. \textit{Bernoulli} 18(1):100--118.
	DOI:10.3150/10-BEJ339
	
	\bibitem{Ios} Iosevich, A., Liflyand, E. (2014). \textit{Decay of the Fourier Transform: Analytic and Geometric
		Aspects}. Basel, Birkh\"{a}user.
	
	\bibitem{IosRud} Iosevich, A., Rudnev, M. (2009). Freiman theorem, Fourier transform and additive structure of measures. \textit{J. Aust. Math. Soc.} 86(1):97--109. DOI:10.1017/S1446788708000530
	
	\bibitem{Adl} Adler, R.J., Taylor, J.E. (2007). \textit{Random Fields and Geometry}. New York, Springer.
	
	\bibitem{Nov} Novikov, D., Schmalzing, J., Mukhanov, V.F. (2000). On non-Gaussianity in the cosmic microwave background. \textit{Astronom. Astrophys.} 364:17-25.
	
	\bibitem{Mar} Marinucci, D. (2004). Testing for non-Gaussianity on cosmic microwave background radiation:
	A review. \textit{Statist. Sci.} 19:294–-307.
	DOI:10.1214/088342304000000783
	
	\bibitem{LeoNew} Leonenko, N., Ruiz-Medina, M.D. (2017). Increasing domain asymptotics for the first Minkowski functional of spherical random fields. \textit{Theory Probab. Math. Statist.} 97:120--141.
	
	\bibitem{DavMart} Davydov, Y.A., Martynova, G.V. (1987). Limit behavior of multiple
	stochastic integral. In \textit{Statistics and control of random process}. Preila, Nauka, Moscow. 55--57 (in Russian).
	
	\bibitem{NourPol} Nourdin, I., Poly, G. (2013). Convergence in total variation on Wiener chaos. \textit{Stoch. Process. Appl.} 123(2):651--674.
	DOI:10.1016/j.spa.2012.10.004
	
	\bibitem{Zint} Zintout, R. (2013). The total variation distance between two double Wiener-It\'{o} integrals. \textit{Stat. Probab. Lett.}, 83(10):2160--2167.
	DOI:10.1016/j.spl.2013.05.030
	
	\bibitem{ReyBla} Reyes, J.B., Blaya, R.A. (2005). Cauchy transform and rectifiability in Clifford analysis. \textit{Z. Anal. Anwendungen} 24(1):167-178.
	DOI:10.4171/ZAA/1235
	
	\bibitem{pec} Peccati, G., Taqqu, M.S. (2011). \textit{Wiener Chaos: Moments, Cumulants and Diagrams. A Survey with Computer Implementation}. Berlin, Springer.
	
	\bibitem{bin} Bingham, N.H., Goldie, C.M., Teugels, J.L. (1987). \textit{Regular Variation.} Cambridge, Cambridge University Press.
	
	\bibitem{leoole} Leonenko, N.N., Olenko, A. (2013). Tauberian and {A}belian theorems for long-range dependent
	random fields. \textit{Methodol. Comput. Appl. Probab.} 15:715--742.
	DOI:10.1007/s11009-012-9276-9
	
	\bibitem{pet} Petrov, V. (1995). \textit{Limit Theorems of Probability Theory}. New York, Oxford Univ. Press.
	
	\bibitem{sen} Seneta, E.  (1976). \textit{Regularly Varying Functions.} Berlin, Springer-Verlag.
	
	\bibitem{LeoLT} Leonenko, N. (1999). \textit{Limit Theorems for Random Fields with Singular Spectrum}. Dordrecht, Kluwer Academic Publishers.
	
	\bibitem{Cor} Corless, R.M., Gonnet, G.H., Hare, D.E., Jeffrey, D.J., Knuth, D.E. (1996). On the LambertW function. \textit{Adv. Comput. Math.} 5(1):329--359.
	DOI:10.1007/BF02124750
	
	\bibitem{Ito} It\'{o}, K. (1951). Multiple Wiener integral. \textit{J. Math. Soc. Japan} 3(1):157--169.
	DOI:10.2969/jmsj/00310157
	
	\bibitem{Maj} Major, P. (1981). \textit{Multiple Wiener-It\'{o} integrals}. Berlin, Springer.
	
	
\end{thebibliography}
\end{document}